\documentclass[11pt,oneside]{amsart}

\usepackage{amsmath,ifthen, amsfonts, amssymb,
srcltx, amsopn, color, enumerate, hyperref}

\usepackage[cmtip,arrow]{xy}
\usepackage{pb-diagram, pb-xy}
\usepackage{overpic}

\usepackage[T1]{fontenc}

\newcommand{\showcomments}{yes}

\newsavebox{\commentbox}
{\ifthenelse{\equal{\showcomments}{yes}}%
{\footnotemark
        \begin{lrbox}{\commentbox}
        \begin{minipage}[t]{1.25in}\raggedright\sffamily\tiny
        \footnotemark[\arabic{footnote}]}
{\begin{lrbox}{\commentbox}}}%
{\ifthenelse{\equal{\showcomments}{yes}}%
{\end{minipage}\end{lrbox}\marginpar{\usebox{\commentbox}}}
{\end{lrbox}}}

\newcounter{intronum}
\newcounter{ax}
\newtheorem{thm}{Theorem}[section]

\newtheorem{lem}[thm]{Lemma}

\newtheorem{cor}[thm]{Corollary}

\newtheorem{prop}[thm]{Proposition}
\newtheorem{thmi}{Theorem}
\newtheorem{cori}[thmi]{Corollary}

\theoremstyle{definition}
\newtheorem*{remnon}{Remark}
\newtheorem{defn}[thm]{Definition}
\newtheorem{rem}[thm]{Remark}

\newtheorem{claim}{Claim}

\newtheorem{claim*}{Claim}

\newtheorem{questioni}[intronum]{Question}

\DeclareMathOperator{\dimension}{dim}
\DeclareMathOperator{\kernel}{ker}

\DeclareMathOperator{\Aut}{Aut}

\DeclareMathOperator{\stabilizer}{Stab}
\DeclareMathOperator{\diam}{diam}

\newcommand{\neb}{\mathcal N}

\newcommand{\symmetric}{\mathrm{Sym}}

\newcommand{\field}[1]{\mathbb{#1}}
\newcommand{\integers}{\ensuremath{\field{Z}}}

\newcommand{\naturals}{\ensuremath{\field{N}}}

\newcommand{\Euclidean}{\ensuremath{\field{E}}}

\makeatletter

\newcommand{\Rmnum}[1]{\mathbf{{\expandafter\@slowromancap\romannumeral #1@}}}

\newcommand{\simp}{\ensuremath{\partial_{_{\triangle}}}}

\makeatother
\DeclareMathOperator{\Isom}{Isom}

\let\oldmarginpar\marginpar
\renewcommand\marginpar[1]{\-\oldmarginpar[\raggedleft\footnotesize #1]
{\raggedright\footnotesize #1}}

\newcounter{enumitemp}

\newcommand{\dist}{\textup{\textsf{d}}}

\newcommand{\cuco}[1]{\mathcal{#1}}

\newcommand{\parint}{\hookrightarrow_{_\parallel}}

\newcommand{\gate}{\mathfrak g}

\newcommand{\orth}[1]{\ensuremath{{#1}^\perp}}

\title{On hierarchical hyperbolicity of cubical groups}
\author[M.F. Hagen]{Mark F Hagen}
\address{DPMMS, University of Cambridge, Cambridge, UK}
\curraddr{School of Mathematics, University of Bristol, Bristol, UK}
\email{markfhagen@gmail.com}
\thanks{\flushleft {Hagen was supported by the Engineering and Physical Sciences Research Council grant of Henry Wilton.}}

\author[T. Susse]{Tim Susse}
\address{Mathematics Department, Bard College at Simon's Rock, Great Barrington, 
Massachusetts, USA}
\email{tisusse@gmail.com}
\thanks{Susse was partially supported by National Science Foundation grant DMS-1313559.}

\date{\today}

\begin{document}
\maketitle

\begin{abstract}
	Let $\cuco X$ be a proper CAT(0) cube complex admitting a proper cocompact 
	action by a group $G$.  We give three conditions on the action, 
any one of which ensures that $\cuco X$ has a \emph{factor system} in the 
	sense of~\cite{BHS:HHS_I}. We also prove that one of these conditions 
is necessary.  This combines with~\cite{BHS:HHS_I} to show that $G$ is 
a \emph{hierarchically hyperbolic group}; this partially answers questions 
raised in~\cite{BHS:HHS_I,BHS:HHS_II}.  Under any of these conditions, our 
results also affirm a conjecture of Behrstock-Hagen on boundaries of cube 
complexes, which implies that $\cuco X$ cannot contain a convex 
\emph{staircase}.  The necessary conditions on the action are all strictly 
weaker than virtual 	cospecialness, and we are not aware of a cocompactly 
cubulated group that does not satisfy at least one of the conditions.
\end{abstract}

\section*{Introduction}\label{sec:introduction}
Much work in geometric group theory revolves around
generalizations of Gromov hyperbolicity: relatively hyperbolic
groups, weakly hyperbolic groups, acylindrically hyperbolic groups,
coarse median spaces, semihyperbolicity, lacunary hyperbolicity,
etc. Much attention has been paid to groups acting
properly and cocompactly on CAT(0) cube complexes, which also
have features reminiscent of hyperbolicity. Such complexes give a
combinatorially and geometrically rich framework to build on, and
many groups have been shown to admit such actions (for a small
sample, see \cite{Sageev:Construction, Wise:small_can,
	OllivierWise:Random, BergeronWise:3mflds, HagenWise:freebyZ}).

Many results follow from studying the geometry of CAT(0) cube
complexes, often using strong properties reminiscent of negative
curvature. For instance, several authors have studied the structure
of quasiflats and Euclidean sectors in cube complexes, with
applications to rigidity properties of right-angled Artin group
\cite{Xie:tits, BKS:quasiflats, Huang:quasiflats}.  These spaces
have also been shown to be median \cite{Chepoi:median} and to have
only semi-simple isometries \cite{Haglund:semisimple}. Further,
under reasonable assumptions, a CAT(0) cube complex $\cuco X$ either
splits as a nontrivial product or $\Isom(\cuco X)$  must contain a
rank-one element~\cite{CapraceSageev:rank}.  Once a given group is known to act properly and cocompactly on a
CAT(0) cube complex the geometry of the cube complex controls the
geometry and algebra of the group. For instance, such a group is
biautomatic and cannot have Kazhdan's property (T)
\cite{NibloReeves:biautomatic, NibloReeves:PropT}, and it must
satisfy a Tits alternative \cite{SageevWise:Tits}

Here, we examine cube complexes admitting
proper, cocompact group actions from the point of view of certain
convex subcomplexes.  Specifically, given a CAT(0) cube complex $\cuco X$, 
we study the following set $\mathfrak F$ of convex subcomplexes: $\mathfrak F$ 
is the smallest set of subcomplexes that contains $\cuco X$, contains each 
combinatorial hyperplane, and is closed under cubical closest-point projection, 
i.e. if $A,B\in\mathfrak F$, then $\gate_B(A)\in\mathfrak F$, where 
$\gate_B:\cuco X\to B$ is the cubical closest point projection.

\subsection*{Main results}
The collection $\mathfrak F$ of subcomplexes is of 
interest for several reasons.  It was first considered in~\cite{BHS:HHS_I}, in 
the context of finding \emph{hierarchically hyperbolic structures} on $\cuco X$. 
Specifically, in~\cite{BHS:HHS_I}, it is shown that if there exists $N<\infty$ 
so that each point of $\cuco X$ is contained in at most $N$ elements of 
$\mathfrak F$, then $\cuco X$ is a \emph{hierarchically hyperbolic space}, which 
has numerous useful consequences outlined below; the same finite multiplicity 
property of $\mathfrak F$ has other useful consequences outlined below.  When 
this finite multiplicity condition holds, we say, following~\cite{BHS:HHS_I}, 
that $\mathfrak F$ is a \emph{factor system} for $\cuco X$.

We believe that if $\cuco X$ is proper and some group $G$ acts properly and 
cocompactly by isometries on $\cuco X$, then the above finite multiplicity 
property holds, and thus $G$ is a hierarchically hyperbolic group.  
In~\cite{BHS:HHS_I}, it is shown that this holds when $G$ has a finite-index 
subgroup acting cospecially on $\cuco X$, and it is also verified in a few 
non-cospecial examples.  

This 
conjecture has proved surprisingly resistant to attack; we earlier believed we 
had a proof.  However, a subtlety 
in Proposition~\ref{prop:closed_uncer_complementation} means that at 
present our techniques only give a complete proof under various conditions on 
the $G$--action, namely:

\begin{thmi}\label{thmi:main}
	Let $G$ act properly and cocompactly on the proper CAT(0) cube complex $\cuco 
	X$.  Then $\mathfrak F$ is a factor system for $\cuco X$ provided 
	\emph{\textbf{any one}} of the following conditions is satisfied (up to passing 
	to a finite-index subgroup of $G$):
	\begin{itemize}
		\item the action of $G$ on $\cuco X$ is \emph{rotational};
		\item the action of $G$ on $\cuco X$ satisfies the \emph{weak finite height 
			condition for hyperplanes};
		\item the action of $G$ on $\cuco X$ satisfies the \emph{essential index 
			condition} and the \emph{Noetherian intersection of conjugates condition 
			(NICC)} on hyperplane-stabilisers.
	\end{itemize}
	Hence, under any of the above conditions, $\cuco X$ is a hierarchically 
	hyperbolic space and $G$ a hierarchically hyperbolic group.
	
	Conversely, if $\mathfrak F$ is a factor system, then the $G$--action satisfies 
	the essential index condition and the NICC.
\end{thmi}

The auxiliary conditions are as 
follows.  The action of $G$ is \emph{rotational} if, whenever $A,B$ are hyperplanes of 
$\cuco X$, and $g\in\stabilizer_G(B)$ has the property that $A$ and $gA$ cross 
or osculate, then $A$ lies at distance at most $1$ from $B$.  This condition is 
\emph{prima facie} weaker than requiring that the 
action of $G$ on $\cuco X$ be cospecial, so Theorem~\ref{thmi:main} generalises 
the results in~\cite{BHS:HHS_I}.  (In fact, the condition above is slightly 
stronger than needed; compare Definition~\ref{defn:forgetful}.)

A subgroup $K\leq G$ satisfies the \emph{weak finite height} condition if the 
following holds.  Let $\{g_i\}_{i\in I}\subset G$ be an infinite set so that 
$K\cap\bigcap_{i\in J}K^{g_i}$ is infinite for all finite $J\subset I$.  Then 
there exist distinct $g_i,g_j$ so that $K\cap K^{g_i}=K\cap K^{g_j}$.  The 
action of $G$ on $\cuco X$ satisfies the weak finite height condition for hyperplanes 
if each hyperplane stabiliser satisfies the weak finite height condition.  

This holds, 
for example, when each hyperplane stabiliser has \emph{finite height} in the 
sense of~\cite{GMRS}.  Hence Theorem~\ref{thmi:main} implies that $\mathfrak F$ 
is a factor system when $\cuco X$ is hyperbolic, without invoking virtual specialness \cite{Agol:virtualhaken} because quasiconvex subgroups 
(in particular hyperplane stabilisers) have finite height~\cite{GMRS}; the 
existence of a hierarchically hyperbolic structure relative to $\mathfrak F$ 
also follows from recent results of Spriano in the hyperbolic 
case~\cite{Shaggy92}. Also, 
if $\mathfrak F$ is a factor system and $\cuco X$ does not decompose as a 
product of unbounded CAT(0) cube complexes, then results of~\cite{BHS:HHS_I} 
imply that $G$ is \emph{acylindrically hyperbolic}.  On the other hand, recent 
work of Genevois~\cite{Genevois:height} uses finite height of 
hyperplane-stabilisers to verify acylindrical hyperbolicity for certain groups 
acting on CAT(0) cube complexes.  In our opinion, this provides some 
justification for the naturality of the weak finite height condition for hyperplanes.  

The NIC condition for hyperplanes asks the following for each 
hyperplane-stabiliser $K$. Given any 
$\{g_i\}_{i\ge0}$ so that 
$K_n=K\cap\bigcap_{i=0}^nK^{g_i}$ is infinite for all $n$, there exists $\ell$ 
so that $K_n$ and $K_\ell$ are commensurable for $n\ge\ell$.  Note that $\ell$ 
is allowed to depend on $\{g_i\}_{i\ge0}$.  The accompanying essential index 
condition asks that there exists a constant $\zeta$ so that for any $F\in\mathfrak F$ , the stabiliser of $F$ has index at most $\zeta$ in the 
stabiliser of the \emph{essential core} of $F$, defined 
in~\cite{CapraceSageev:rank}.  These conditions are somewhat less natural than 
the preceding conditions, but they follow fairly easily from the finite 
multiplicity of $\mathfrak F$.

We prove Theorem~\ref{thmi:main} in Section~\ref{sec:proof_of_main_theorem}.  There is a unified argument under the weak finite 
height and NICC hypotheses, and a somewhat simpler argument in the presence of 
a rotational action.

To prove Theorem~\ref{thmi:main}, the main issue is to derive a contradiction 
from the existence of an infinite strictly ascending chain $\{F_i\}$, in 
$\mathfrak F$, using that the corresponding chain of orthogonal complements 
must strictly descend.  The existence of such chains can be deduced from the 
failure of the finite multiplicity of $\mathfrak F$ using only the proper 
cocompact group action; it is in deriving a contradiction from the existence of 
such chains that the other conditions arise.

\textbf{Any condition} that allows one to conclude that the $F_i$ 
have bounded-diameter fundamental domains for the actions of their stabilisers yields the desired conclusion.  So, there 
are most likely other versions of Theorem~\ref{thmi:main} using different 
auxiliary 
hypotheses. We are not aware of a cocompactly cubulated group which is not covered 
by Theorem~\ref{thmi:main}.

\subsection*{Hierarchical hyperbolicity}
\emph{Hierarchically hyperbolic
	spaces/groups} (HHS/G's), introduced in~\cite{BHS:HHS_I,BHS:HHS_II}, were proposed as a common framework for studying mapping
class groups and (certain) cubical groups. Knowledge that a group is
hierarchically hyperbolic has strong consequences that imply 
many of the nice properties of mapping class groups. 

Theorem~\ref{thmi:main} and results of~\cite{BHS:HHS_I} (see Remark 13.2 of 
that 
paper) together answer Question 8.13 
of~\cite{BHS:HHS_I} and part of Question~A of~\cite{BHS:HHS_II} --- which ask 
whether a proper cocompact CAT(0) cube complex has a factor system --- under 
any of the three auxiliary hypotheses in Theorem~\ref{thmi:main}.  Hence our 
results expand the class of cubical groups that are known to be 
hierarchically hyperbolic.  Some consequences of this are as follows, where 
$\cuco X$ is a CAT(0) cube complex on which $G$ acts geometrically, satisfying 
any of the hypotheses in Theorem~\ref{thmi:main}:

\begin{itemize}
	\item In combination with~\cite[Corollary 14.5]{BHS:HHS_I}, 
	Theorem~\ref{thmi:main} shows that $G$ acts acylindrically on the 
contact graph 
	of $\cuco X$, i.e. the intersection graph of the hyperplane carriers, which is 
	a quasi-tree~\cite{Hagen:contact}.
	\item Theorem~\ref{thmi:main} combines with Theorem~9.1 
of~\cite{BHS:HHS_I} to 
	provide a Masur-Minsky style distance estimate in $G$: up to quasi-isometry, 
	the distance in $\cuco X$ from $x$ to $gx$, where $g\in G$, is given by summing 
	the distances between the projections of $x,gx$ to a collection of uniform 
	quasi-trees associated to the elements of the factor system.
	\item Theorem~\ref{thmi:main} combines with Corollary~9.24 
of~\cite{DHS:HHS_IV} 
	to prove that either $G$ stabilizes a convex subcomplex of $\cuco X$ splitting 
	as the product of unbounded subcomplexes, or $G$ contains an element acting 
	loxodromically on the contact graph of $\cuco X$.  This is a new proof of a 
	special case  of the Caprace-Sageev rank-rigidity 
	theorem~\cite{CapraceSageev:rank}.
\end{itemize}

Proposition~11.4 of~\cite{BHS:HHS_I} combines with Theorem~\ref{thmi:main} to 
prove:

\begin{thmi}\label{thmi:quasi_trees}
	Let $G$ act properly and cocompactly on the proper CAT(0) cube complex $\cuco 
	X$, with the action satisfying the hypotheses of Theorem~\ref{thmi:main}.  Let 
	$\mathfrak F$ be the factor system, and suppose that for all 
	subcomplexes $A\in\mathfrak F$ and $g\in G$, the subcomplex $gA$ is not 
	parallel 
	to a subcomplex in $\mathcal F$ which is in the orthogonal complement of $A$.  
	Then $\cuco X$ quasi-isometrically embeds in the product of finitely many trees.
\end{thmi}

The set $\mathfrak F$ is shown in Section~\ref{sec:decomposing} to have a
graded structure: the lowest-grade elements are combinatorial
hyperplanes, then we add projections of combinatorial hyperplanes to
combinatorial hyperplanes, etc. This allows for several arguments to
proceed by induction on the grade. Essentially by definition, a
combinatorial hyperplane $H$ is the \emph{orthogonal complement} of
a $1$--cube $e$, i.e. a maximal convex subcomplex $H$ for which
$\cuco X$ contains the product $e\times H$ as a subcomplex.  We
show, in Theorem~\ref{thm:factor_system_compact_orthocomp}, that $\mathfrak F$ is precisely the set of convex subcomplexes $F$ such
that there exists a compact, convex subcomplex $C$ so that the
orthogonal complement of $C$ is $F$.  This observation plays an important role.

Relatedly, we give conditions in Proposition~\ref{prop:closed_uncer_complementation} ensuring that $\mathfrak F$ is closed under the operation of taking orthogonal 
complements.  As well as being used in the proof of Theorem~\ref{thmi:main}, 
this is needed for applications of recent results about hierarchically 
hyperbolic spaces to cube complexes.  Specifically, in~\cite{ABD}, 
Abbott-Behrstock-Durham introduce hierarchically hyperbolic 
spaces with \emph{clean containers}, and work under that (quite natural) 
hypothesis.  Among its applications, they produce largest, universal 
acylindrical actions on hyperbolic spaces for hierarchically hyperbolic groups.  
We will not give the definition of clean containers for general hierarchically 
hyperbolic structures, but for the CAT(0) cubical case, our results imply that 
it holds for hierarchically hyperbolic structures on $\cuco X$ obtained using 
$\mathfrak F$, as follows:

\begin{thmi}[Clean containers]\label{thmi:clean_containers}
	Let $\cuco X$ be a proper CAT(0) cube complex on which the group $G$ acts 
	properly and cocompactly, and suppose $\mathfrak F$ is a factor system.  Let $F\in\mathfrak F$, and let $V\in\mathfrak F$ be 
	properly contained in $F$.  Then there exists $U\in\mathfrak F$, unique 
up to parallelism, such 	that:
	\begin{itemize}
		\item $U\subset F$;
		\item $V\hookrightarrow F$ extends to a convex embedding $V\times 
		U\hookrightarrow F$;
		\item if $W\in\mathfrak F$, and the above two conditions hold with $U$ 
		replaced by $W$, then $W$ is parallel to a subcomplex of $U$.
	\end{itemize}
\end{thmi}

\begin{proof}
	Let $x\in V$ be a $0$--cube and 
	let $U'=\orth V$, the orthogonal complement of $V$ at $x$ (see 
	Definition~\ref{defn:orthocomp}).  
Proposition~\ref{prop:closed_uncer_complementation}
	implies that $U'\in\mathfrak F$, so $U=U'\cap F$ is also in $\mathfrak 
F$, since $\mathfrak F$ is closed under projections.  By the definition of the 
orthogonal complement, $V\to\cuco X$ extends to a convex embedding $V\times 
U'\to \cuco X$, and $(V\times U')\cap F=V\times U$ since $V\subset F$ and 
$F,V\times U'$ are convex.  Now, if $W\in\mathfrak F$ and $W\subset F$, and 
$V\to\cuco X$ extends to a convex embedding $V\times W\to\cuco X$, then 
$V\times W$ is necessarily contained in $F$, by convexity.  On the other hand, 
by the definition of the orthogonal complement, $W$ is parallel to a subcomplex 
of $U'$.  Hence $W$ is parallel to a subcomplex of $U$.  This implies the third 
assertion and uniqueness of $U$ up to parallelism.
\end{proof}

We now turn to applications of Theorem~\ref{thmi:main} that do not involve 
hierarchical hyperbolicity.

\subsection*{Simplicial boundary and staircases}
Theorem~\ref{thmi:main} also gives insight into the structure of the boundary 
of $\cuco X$.  We first mention an aggravating 
geometric/combinatorial question about cube complexes which is partly answered 
by our results.

A \emph{staircase} is a CAT(0) cube complex $\cuco Z$ defined as
follows.  First, a \emph{ray-strip} is a square complex of the form
$S_n=[n,\infty)\times[-\frac12,\frac12]$, with the product
cell-structure where $[n,\infty)$ has $0$--skeleton
$\{m\in\integers:m\geq n\}$ and $[-\frac12,\frac12]$ is a $1$--cube.
To build $\cuco Z$, choose an increasing sequence
$(a_n)_n$ of integers, collect the ray-strips
$S_{a_n}\cong[a_n,\infty)\times[-\frac12,\frac12]$, and identify
$[a_{n+1},\infty)\times\{-\frac12\}\subset S_{a_n+1}$ with
$[a_{n},\infty)\times\{\frac12\}\subset S_{a_n}$ for each $n$.  The
model staircase is the cubical neighbourhood of a Euclidean sector in
the standard tiling of $\Euclidean^2$ by squares, with one bounding
ray in the $x$--axis, but for certain $(a_n)_n$,
$\cuco Z$ may not contain a nontrivial Euclidean sector. (One can define a 
$d$-dimensional staircase analogously for $d\ge2$.)
%
%
We will see below that the set of ``horizontal'' hyperplanes in $\cuco Z$ 
--
see Figure~\ref{fig:stairs} for the meaning of ``horizontal'' -- is interesting 
because there is no geodesic ray in $\cuco Z$ crossing exactly the set of 
horizontal
hyperplanes.

\begin{figure}[h]
	\includegraphics[width=0.15\textwidth]{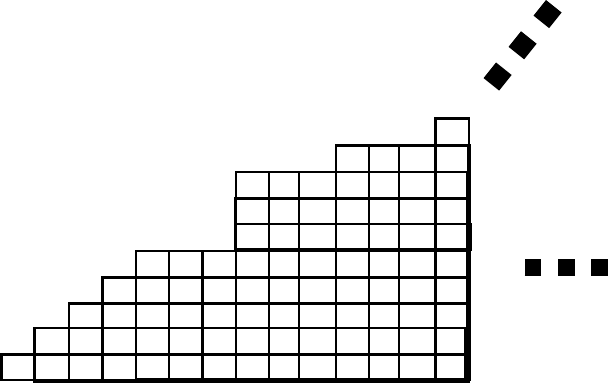}
	\caption{Part of a staircase.}\label{fig:stairs}
\end{figure}

Now let $\cuco X$ be a proper CAT(0) cube complex with a group $G$ acting 
properly and cocompactly.  Can there be a convex staircase subcomplex in $\cuco 
X$?  A positive answer seems very implausible, but this question 
is open and has bothered numerous researchers.

In Section~\ref{sec:fullvis}, we prove that if $\mathfrak F$ is a factor 
system, then $\cuco X$ cannot contain a convex staircase.  Hence, if $\cuco X$ 
admits a geometric group action satisfying any of the hypotheses in 
Theorem~\ref{thmi:main}, then $\cuco X$ cannot contain a convex staircase.  In 
fact, we prove something more general, which is best formulated in terms of the 
\emph{simplicial boundary} $\simp\cuco X$.

Specifically, the \emph{simplicial boundary} $\simp\cuco X$ of a CAT(0) cube 
complex $\cuco X$
was defined in~\cite{Hagen:boundary}. Simplices of $\simp\cuco
X$ come from equivalence classes of infinite sets $\mathcal H$ of 
hyperplanes such that:
\begin{itemize}
	\item if $H,H'\in\mathcal H$ are separated by a hyperplane $V$, then 
	$V\in\mathcal H$;
	\item if $H_1,H_2,H_3\in\mathcal H$ are disjoint, then one of $H_1,H_2,H_3$ 
	separates the other two;
	\item for $H\in\mathcal H$, at most one halfspace associated to $H$ contains 
	infinitely many $V\in\mathcal H$.
\end{itemize}
\noindent These \emph{boundary sets} are partially ordered by coarse
inclusion (i.e., $A\preceq B$ if all but finitely many hyperplanes
of $A$ are contained in $B$), and two are equivalent if
they have finite symmetric difference; $\simp\cuco X$ is the
simplicial realization of this partial order. The motivating example of a 
simplex of $\simp\cuco X$ is: given a
geodesic ray $\gamma$ of $\cuco X$, the set of hyperplanes crossing
$\gamma$ has the preceding properties.  Not all simplices are realized by a 
geodesic ray in this way: a simplex in $\cuco
X$ is called \emph{visible} if it is.  For example, if 
$\cuco 
Z$ is a staircase, then $\simp\cuco Z$ has an invisible $0$--simplex, 
represented by the set of horizontal hyperplanes.

Conjecture~2.8 of~\cite{BH:thick} holds that every simplex of $\simp\cuco X$ is 
visible when $\cuco X$ admits a proper cocompact group action; Theorem~\ref{thmi:main} hence proves a special case.  Slightly more generally:

\begin{cori}\label{cori:visible}Let $\cuco X$ be a proper CAT(0) cube complex 
	which admits a proper and cocompact group action satisfying the NICC for 
	hyperplanes. Then every simplex is 
	$\simp\cuco X$ is visible.  Moreover, let $v\in\simp\cuco X$ be a 
	$0$--simplex.  Then there exists a CAT(0) geodesic ray $\gamma$ such that the 
	set of hyperplanes crossing $\gamma$ represents $v$.
\end{cori}

The above could, in principle, hold even if $\mathfrak F$ is not a 
factor system, since we have not imposed the essential index condition.  The ``moreover'' part follows from the first part 
and~\cite[Lemma~3.32]{Hagen:boundary}. Corollary~\ref{cori:visible} combines 
with \cite[Theorem 5.13]{BH:thick} to imply that $\simp\cuco X$ detects 
thickness of order 1 and quadratic divergence for $G$, under the NICC 
condition. Corollary~\ref{cori:visible} also implies the corollary about 
staircases at the beginning of this paper.  More generally, we obtain the 
following from Corollary~\ref{cori:visible} and a simple argument 
in~\cite{Hagen:boundary}:

\begin{cori}\label{cori:orthant}
	Let $\gamma$ be a CAT(0)-metric or combinatorial geodesic ray in
	$\cuco X$, where $\cuco X$ is as in Corollary~\ref{cori:visible} and
	the set of hyperplanes crossing $\gamma$ represents a
	$d$-dimensional simplex of $\simp\cuco X$. Then there exists a
	combinatorially isometrically embedded $d+1$-dimensional orthant
	subcomplex $\mathcal O\subseteq Hull(\gamma)$. Moreover, $\gamma$
	lies in a finite neighbourhood of $\mathcal O$.
\end{cori}

(A $k$-dimensional orthant subcomplex is a CAT(0) cube complex
isomorphic to the product of $k$ copies of the standard tiling of
$[0,\infty)$ by $1$--cubes, and the convex hull $Hull(A)$ of a
subspace $A\subseteq\cuco X$ is the smallest convex subcomplex
containing $A$.)

Corollary~\ref{cori:orthant} is related to Lemma~4.9
of~\cite{Huang:quasiflats} and to statements
in~\cite{Xie:tits,BKS:quasiflats} about Euclidean sectors in
cocompact CAT(0) cube complexes and arcs in the Tits
boundary. In particular it shows that in any CAT(0) cube complex with a 
proper cocompact group action satisfying NICC, nontrivial geodesic arcs on the 
Tits boundary extend to arcs of length $\pi/2$.

\subsection*{Further questions and approaches}
We believe that any proper cocompact CAT(0) cube complex admits a factor 
system, but that some additional ingredient is needed to remove the auxiliary 
hypotheses in Theorem~\ref{thmi:main}; we hope that the applications we have 
outlined stimulate interest in finding this additional idea.  Since the 
property of admitting a factor system is inherited by convex 
subcomplexes~\cite{BHS:HHS_I}, we suggest trying to 
use $G$--cocompactness of $\cuco X$ to arrange for a convex 
(non-$G$--equivariant) embedding of $\cuco X$ into a CAT(0) cube complex $\cuco 
Y$ where a factor system can be more easily shown to exist.  One slightly 
outrageous possibility is:

\begin{questioni}\label{question:A}
	Let $\cuco X$ be a CAT(0) cube complex which admits a proper and cocompact 
	group action. Does $\cuco X$ embed as a convex subcomplex of the universal cover 
	of the Salvetti complex of some right-angled Artin group?
\end{questioni}

However, there are other possibilities, for example trying to embed $\cuco X$ 
convexly in a CAT(0) cube complex whose automorphism group is sufficiently tame 
to enable one to use the proof of Theorem~\ref{thmi:main}, or some variant of 
it.

\subsection*{Plan of the paper}
In Section~\ref{sec:background}, we discuss the necessary background.  
Section~\ref{sec:decomposing} contains basic facts about $\mathfrak F$, and 
Section~\ref{sec:orthocomp} relates $\mathfrak F$ to orthogonal complements.  
Section~\ref{sec:uniform} introduces the auxiliary hypotheses for 
Theorem~\ref{thmi:main}, which we prove in 
Section~\ref{sec:proof_of_main_theorem}.  The applications to the simplicial 
boundary are discussed in Section~\ref{sec:fullvis}.

\subsection*{Acknowledgements}
MFH thanks: Jason Behrstock and Alessandro Sisto for discussions of factor 
systems during our work on~\cite{BHS:HHS_I}; Nir Lazarovich and Dani Wise for 
discussions on Question~\ref{question:A}; Talia Fern\'os, 
Dan Guralnik, Alessandra Iozzi, Yulan Qing, and Michah Sageev for discussions 
about staircases.  We thank Richard Webb and Henry Wilton for helping to 
organize \emph{Beyond Hyperbolicity} (Cambridge, June 2016), at 
which much of the work on this paper was done.  Both of the authors thank 
Franklin for his sleepy vigilance and Nir Lazarovich for a comment on the proof 
of Lemma~\ref{lem:chain_find}.  We are greatly indebted to Bruno 
Robbio and Federico Berlai for a discussion which led us to discover a gap in 
an earlier version of this paper, and to Elia Fioravanti for independently 
drawing our attention to the same issue.  We also thank an anonymous referee 
for helpful comments on an earlier version, and another anonymous referee for several useful comments, including suggesting 
a simplification of an argument in Section~\ref{subsec:closed_under_complementation}.

\section{Background}\label{sec:background}

\subsection{Basics on CAT(0) cube complexes}\label{subsec:cubes_convexity}
Recall that a \emph{CAT(0) cube complex} $\cuco X$ is a
simply-connected cube complex in which the link of every vertex is a
simplicial flag complex (see e.g.~\cite[Chapter II.5]{Bridson:NPC},~\cite{Sageev:PCMI,Wise:QCH,Chepoi:median}
for precise definitions and background).  In this paper, $\cuco X$
always denotes a CAT(0) cube complex.  Our choices of language and
notation for describing convexity, hyperplanes, gates, etc. follow
the account given in~\cite[Section 2]{BHS:HHS_I}.

\begin{defn}[Hyperplane, carrier, combinatorial hyperplane]\label{defn:hyperplane}
A \emph{midcube} in the unit cube $c=[-\frac12,\frac12]^n$ is a subspace
obtained by restricting exactly one coordinate to $0$.  A
\emph{hyperplane} in $\cuco X$ is a connected subspace $H$ with the
property that, for all cubes $c$ of $\cuco X$, either $H\cap
c=\emptyset$ or $H\cap c$ consists of a single midcube of $c$. The
\emph{carrier} $\neb(H)$ of the hyperplane $H$ is the union of all
closed cubes $c$ of $\cuco X$ with $H\cap c\neq\emptyset$. The
inclusion $H\to\cuco X$ extends to a combinatorial embedding
$H\times[-\frac12,\frac12]\stackrel{\cong}{\longrightarrow}\neb(H)\hookrightarrow\cuco
X$ identifying $H\times\{0\}$ with $H$.  Now, $H$ is isomorphic to a
CAT(0) cube complex whose cubes are the midcubes of the cubes in
$\neb(H)$.  The subcomplexes $H^\pm$ of $\neb(H)$ which are the
images of $H\times\{\pm\frac12\}$ under the above map are isomorphic
as cube complexes to $H$, and are \emph{combinatorial hyperplanes}
in $\cuco X$.  Thus each hyperplane of $\cuco X$ is associated to
two combinatorial hyperplanes lying in $\neb(H)$.
\end{defn}

\begin{remnon}\label{rem:comb_hyp}
The distinction between hyperplanes (which are not subcomplexes) and combinatorial hyperplanes (which are) is important.  Given $A\subset\cuco X$, either a convex subcomplex or a hyperplane, and a hyperplane $H$, we sometimes say $H$ \emph{crosses} $A$ to mean that $H\cap A\neq\emptyset$.  Observe that the set of hyperplanes crossing a hyperplane $H$ is precisely the set of hyperplanes crossing the associated combinatorial hyperplanes.
\end{remnon}

\begin{defn}[Convex subcomplex]\label{defn:convex}
A subcomplex $\cuco Y\subseteq\cuco X$ is \emph{convex} if $\cuco Y$
is \emph{full} --- i.e. every cube $c$ of $\cuco X$ whose
$0$--skeleton lies in $\cuco Y$ satisfies $c\subseteq\cuco Y$ ---
and $\cuco Y^{(1)}$, endowed with the obvious path-metric, is
metrically convex in $\cuco X^{(1)}$. \end{defn}

There are various
characterizations of cubical convexity.  Cubical convexity coincides
with CAT(0)--metric convexity \emph{for
subcomplexes}~\cite{Haglund:semisimple}, but not for arbitrary
subspaces.

\begin{defn}[Convex Hull] Given a subset $A\subset\cuco X$, we denote by $Hull(A)$ its
\emph{convex hull}, i.e. the intersection of all convex subcomplexes
containing $A$.
\end{defn}

If $\cuco Y\subseteq\cuco X$ is a convex subcomplex, then $\cuco Y$
is a CAT(0) cube complex whose hyperplanes have the form $H\cap\cuco
Y$, where $H$ is a hyperplane of $\cuco X$, and two hyperplanes
$H\cap\cuco Y,H'\cap\cuco Y$ intersect if and only if $H,H'$
intersect. 

Recall from~\cite{Chepoi:median} that the graph $\cuco X^{(1)}$,
endowed with the obvious path metric $\dist_{\cuco X}$ in which
edges have length $1$, is a \emph{median graph} (and in fact being a
median graph characterizes $1$--skeleta of CAT(0) cube complexes
among graphs): given $0$--cubes $x,y,z$, there exists a unique
$0$--cube $m=m(x,y,z)$, called the \emph{median} of $x,y,z$, so that
$Hull(x,y)\cap Hull(y,z)\cap Hull(x,z)=\{m\}$.

Let $\cuco Y\subseteq\cuco X$ be a convex subcomplex.  Given a
$0$--cube $x\in\cuco X$, there is a unique $0$--cube $y\in\cuco Y$
so that $\dist_{\cuco X}(x,y)$ is minimal among all $0$--cubes in
$\cuco Y$.  Indeed, if $y'\in\cuco Y$, then the median $m$ of
$x,y,y'$ lies in $\cuco Y$, by convexity of $\cuco Y$, but
$\dist_{\cuco X}(x,y')=\dist_{\cuco X}(x,m)+\dist_{\cuco X}(m,y')$,
and the same is true for $y$.  Thus, if $\dist_{\cuco X}(x,y')$ and
$\dist_{\cuco X}(x,y)$ realize the distance from $x$ to $\cuco
Y^{(0)}$, we have $m=y=y'$.

\begin{defn}[Gate map on $0$--skeleton]\label{defn:gate_map} For a convex subcomplex $\cuco Y\subseteq \cuco X$, the \emph{gate map} to $\cuco Y$ is the map $\gate_{\cuco Y}:\cuco
X^{(0)}\to\cuco Y^{(0)}$ so that or all $v\in \cuco
X^{(0)}$, $\gate_{\cuco Y}(v)$ is the unique $0$--cube of $\cuco Y$
lying closest to $v$.
\end{defn}

\begin{lem}\label{lem:cubical_gate_map}
Let $\cuco Y\subseteq\cuco X$ be a convex subcomplex.  Then the map $\gate_{\cuco Y}$ from Definition~\ref{defn:gate_map} extends to a unique cubical map $\gate_{\cuco Y}:\cuco X\to\cuco Y$ so that the following holds:  for any $d$--cube $c$, of
$\cuco X$ with vertices $x_0,\ldots,x_{2^d}\in\cuco X^{(0)}$, the map
$\gate_{\cuco Y}$ collapses $c$ to the unique $k$--cube $c'$ in $\cuco Y$
with $0$--cells $\gate_{\cuco Y}(x_0)\ldots, \gate_{\cuco
Y}(x_{2^d})$ in the natural way, respecting the cubical
structure.  

Furthermore, for any convex subcomplex
$\cuco Y, \cuco Y'\subseteq \cuco X$, the hyperplanes crossing
$\gate_{\cuco Y}(\cuco Y')$ are precisely the hyperplanes which
cross both $\cuco Y$ and $\cuco Y'$.
\end{lem}

\begin{proof}
The first part is proved in~\cite[p. 1743]{BHS:HHS_I}: observe that the integer $k$ is the number of
hyperplanes that intersect both $c$ and $\cuco Y$. The hyperplanes that intersect  $c'$ are precisely the
hyperplanes which intersect both $c$ and $\cuco Y$.  Indeed, the Helly property ensures that there are cubes crossing exactly this set of hyperplanes, while convexity of $\cuco Y$ shows that at least one such cube lies in $\cuco Y$; the requirement that it contain $\gate_{\cuco Y}(x_i)$ then uniquely determines $c'$.

To prove the second statement, let $H$ be a hyperplane crossing $\cuco Y$ and $\cuco Y'$.  Then $H$ separates $0$--cubes $y_1,y_2\in\cuco Y'$, and thus separates their gates in $\cuco Y$, since, because it crosses $\cuco Y$, it cannot separate $y_1$ or $y_2$ from $\cuco Y$.  On the other hand, if $H$ crosses $\gate_{\cuco Y}(\cuco Y')$, then it separates $\gate_{\cuco Y}(y_1),\gate_{\cuco Y}(y_2)$ for some $y_1,y_2\in\cuco Y'$.  Since it cannot separate $y_1$ or $y_2$ from $\cuco Y$, the hyperplane $H$ must separate $y_1$ from $y_2$ and thus cross $\cuco Y$.  (Here we have used the standard fact that $H$ separates $y_i$ from $\gate_{\cuco Y}(y_i)$ if and only if $H$ separates $y_i$ from $\cuco Y$; see e.g.~\cite[p. 1743]{BHS:HHS_I}.)  Hence $H$ crosses $\gate_{\cuco Y}(\cuco Y')$ if and only if $H$ crosses $\cuco Y,\cuco Y'$.
\end{proof}

The next definition formalizes the relationship between $\gate_{\cuco Y}(\cuco Y'),\gate_{\cuco Y'}(\cuco Y)$ in the above lemma.

\begin{defn}[Parallel]\label{defn:parallel}
The convex subcomplexes $F$ and $F'$ are \emph{parallel}, written $F\parallel F'$, if for each hyperplane $H$ of $\cuco X$, we have $H\cap F\neq\emptyset$ if and only if $H\cap F'\neq\emptyset$.  The subcomplex $F$ is \emph{parallel into} $F'$ if $F$ is parallel to a subcomplex of $F'$, i.e. every hyperplane intersecting $F$ intersects $F'$.  We denote this by $F\parint F'$.  Any two $0$--cubes are parallel subcomplexes.
\end{defn}

The following is proved in~\cite[Section 2]{BHS:HHS_I} and illustrated in Figure~\ref{fig:parallel}:

\begin{lem}\label{lem:parallel_product}
Let $F,F'$ be parallel subcomplexes of the CAT(0) cube complex
$\cuco X$.  Then $Hull(F\cup F')\cong F\times A$, where $A$ is the convex hull of a shortest
combinatorial geodesic with endpoints on $F$ and $F'$.  The
hyperplanes intersecting $A$ are those separating $F,F'$.
Moreover, if $D,E\subset\cuco X$ are convex subcomplexes, then
$\gate_E(D)\subset E$ is parallel to $\gate_D(E)\subset D$.
\end{lem}

\begin{figure}[h]
\begin{overpic}[width=0.35\textwidth]{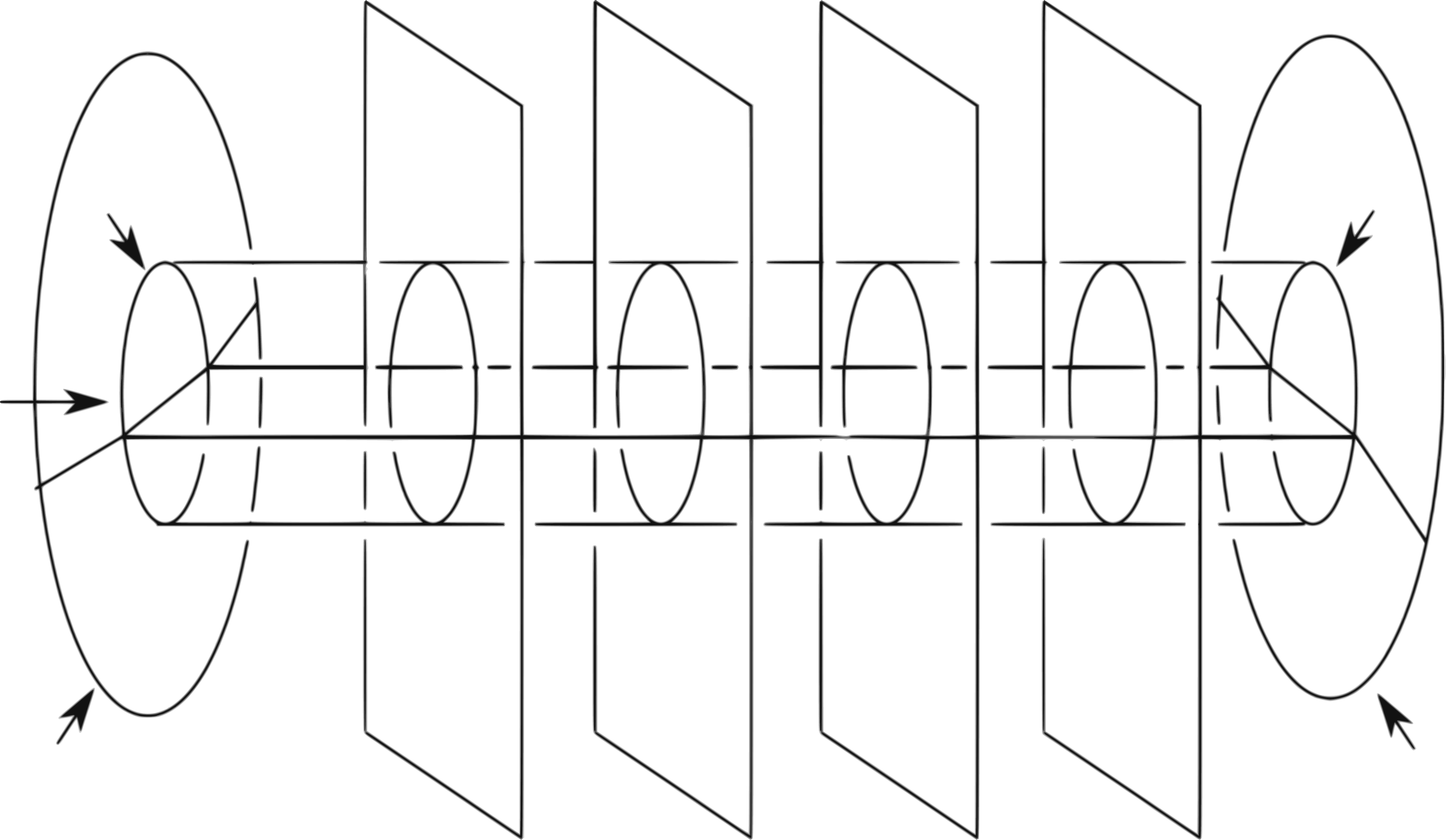}
\put(0,1){$D$}
\put(97,2){$E$}
\put(-6,45){$\gate_{D}(E)$}
\put(95,45){$\gate_{E}(D)$}
\put(-6,29){$H$}
\put(26,-2){$V$}
\end{overpic}
\caption{Here, $D,E$ are convex subcomplexes.  The gates $\gate_D(E),\gate_{E}(D)$ are parallel, and are joined by a product region, shown as a cylinder.  Each hyperplane crossing $Hull\big(\gate_D(E)\cup\gate_{E}(D)\big)$ either separates $\gate_D(E),\gate_{E}(D)$ (e.g. the hyperplane $V$) or crosses both of $\gate_D(E),\gate_{E}(D)$ (e.g. the hyperplane $H$).}\label{fig:parallel}
\end{figure}

The next Lemma will be useful in Section~\ref{sec:decomposing}.

\begin{lem}\label{lem:project_curry}For convex subcomplexes $C,D,E$, we have $\gate_{\gate_C(D)}(E)\parallel\gate_C(\gate_D(E))\parallel\gate_C(\gate_E(D))$.
\end{lem}

\begin{proof}
Let $F=\gate_C(D)$.  Let $H$ be a hyperplane so that $H\cap \gate_F(E)\neq\emptyset$.
Then $H\cap E,H\cap F\neq \emptyset$ and thus $H\cap C, H\cap D\neq
\emptyset$, by Lemma~\ref{lem:cubical_gate_map}.
Thus $\gate_F(E)$ is parallel into $\gate_C(\gate_D(E))$ and
$\gate_C(\gate_E(D))$. However, the hyperplanes crossing either of
these are precisely the hyperplanes crossing all of $C,D,E$. Thus,
they cross $F$ and $D$, and thus cross $\gate_F(E)$ by Lemma~\ref{lem:cubical_gate_map}.
\end{proof}

The next lemma will be used heavily in Section~\ref{sec:proof_of_main_theorem}, and gives a group theoretic description of the stabilizer of a projection.

\begin{lem}\label{lem:stab_proj} Let $\cuco X$ be a proper CAT(0) cube 
complex on which $G$ acts properly and cocompactly.  Let $H$, $H'$ be two 
convex subcomplexes in $\cuco X$ such that $\stabilizer_G(H)$ acts cocompactly on $H$ and $\stabilizer_G(H')$ acts 
cocompactly on $H'$. Then $\stabilizer_G(\gate_H(H'))$ is 
commensurable with $\stabilizer_G(H)\cap\stabilizer_G(H')$. Further, for any 
finite collection $H_1, \ldots, H_n$ of convex subcomplexes whose stabilisers 
act cocompactly, 
$\stabilizer_G(\gate_{H_1}(\gate_{H_2}(\cdots\gate_{H_{n-1}}(H_n)\cdots)))$ is 
commensurable with $\bigcap_{i=1}^n \stabilizer_G(H_i)$.\end{lem}
\begin{proof} Let $H$ and $H'$ be two convex subcomplexes and suppose that $g\in\stabilizer_G(H)\cap\stabilizer_G(H')$. Then $g\in \stabilizer_G(\gate_H(H'))$, and thus $\stabilizer_G(H)\cap\stabilizer_G(H')\le\stabilizer_G(\gate_H(H'))$.
	
Let $d=d(H,H')$. In particular, for any $0$--cube in $\gate_H(H')$, its distance 
to $H'$ is exactly $d$. However, there are only finitely many such translates 
of $H'$ in $\cuco X$, and any element of $\stabilizer_G(\gate_H(H'))$ must 
permute these. Further, there are only finitely many translates of $H$ in 
$\cuco X$ that contain $\gate_H(H')$, and any element of the stabilizer must 
also permute those. Thus, there is a finite index subgroup (obtained as the 
kernel of the permutation action on the finite sets of hyperplanes) that 
stabilizes both $H$ and $H'$.  A similar argument covers the case of finitely many complexes.
\end{proof}

\begin{defn}[Orthogonal complement]\label{defn:orthocomp}
Let $A\subseteq\cuco X$ be a convex subcomplex.  Let $P_A$ be the
convex hull of the union of all parallel copies of $A$, so that
$P_A\cong A\times \orth A$, where $\orth A$ is a CAT(0) cube complex
that we call the \emph{abstract orthogonal complement of $A$ in
$\cuco X$}.  Let $\phi_A:A\times\orth A\to\cuco X$ be the cubical isometric embedding with image $P_A$. 

For any $a\in A^{(0)}$, the convex subcomplex
$\phi_A(\{a\}\times\orth A)$ is the \emph{orthogonal complement of $A$ at
$a$}.  See Figures~\ref{fig:orthocomp_1} and~\ref{fig:orthocomp_2}.
\end{defn}

\begin{figure}[h]
\begin{overpic}[width=0.15\textwidth]{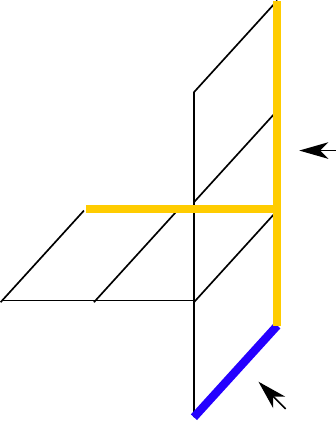}
\put(70,0){$e$}
\put(80,62){$\orth e$}
\end{overpic}

\caption{Combinatorial hyperplanes are orthogonal complements of $1$--cubes.}\label{fig:orthocomp_1}
\end{figure}

\begin{figure}[h]
\begin{overpic}[width=0.3\textwidth]{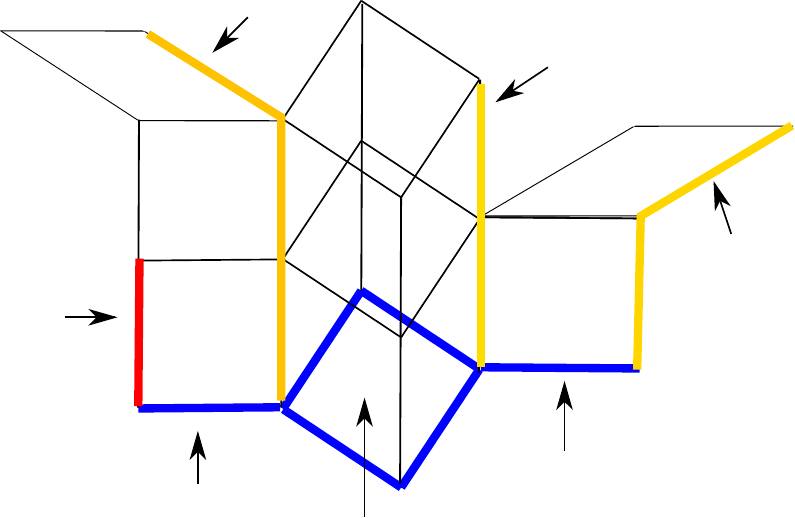}
\put(70,5){$e_2$}
\put(-35,24){$\orth{(e_1\cup s\cup e_2)}$}
\put(23,0){$e_1$}
\put(44,-3){$s$}
\put(69,56){$\orth s$}
\put(91,33){$\orth e_2$}
\put(32,63){$\orth e_1$}
\end{overpic}

\caption{Orthogonal complements of $1$--cubes $e_1,e_2$ and $2$--cube $s$ are 
shown.  Note that $\orth{(e_1\cup e_2\cup s)}\parallel\gate_{\orth 
e_2}(\gate_{\orth e_1}(\orth s))$.}\label{fig:orthocomp_2}
\end{figure}

\begin{lem}\label{lem:orthocomp}
Let $A\subseteq\cuco X$ be a convex subcomplex.  For any $a\in A$, a hyperplane $H$ intersects $\phi_A(\{a\}\times\orth A)$ if and only if $H$ is disjoint from every parallel copy of $A$ but intersects each hyperplane $V$ with $V\cap A\neq\emptyset$.  Hence $\phi_A(\{a\}\times \orth A),\phi_A(\{b\}\times\orth A)$ are parallel for all $a,b\in A^{(0)}$.
\end{lem}

\begin{proof}
This follows from the definition of $P_A$: the hyperplanes crossing $P_A$ are partitioned into two classes, those intersecting $A$ (and its parallel copies) and those disjoint from $A$ (and any of its parallel copies).  By definition, $\phi_A(\{a\}\times\orth A)$ is the convex hull of the set of $0$--cubes of $P_A$ that are separated from $a$ only by hyperplanes of the latter type.  The product structure ensures that any hyperplane of the first type crosses every hyperplane of the second type.
\end{proof}

Finally, in \cite{CapraceSageev:rank}, Caprace and Sageev defined the notion of 
an essential hyperplane and an essential action. We record the necessary facts 
here.

\begin{defn}Let $\cuco X$ be a CAT(0) cube complex, and let $F\subseteq\cuco X$ be a convex subcomplex. Let $G\le\text{Aut}(\cuco X)$ preserve $F$.
	\begin{enumerate}
		\item We say that a hyperplane $H$ is \emph{essential} in $F$ if $H$ crosses $F$, and each halfspace associated to $H$ contains $0$--cubes of $F$ which are arbitrarily far from $H$.
		\item We say that $H$ is \emph{$G$--essential} in $F$ if for any $0$--cube $x\in F$, each halfspace 
associated to $H$ contains elements of $Gx$ arbitrarily far from $H$.
		\item We say that $G$ acts \emph{essentially} on $F$ if every hyperplane crossing $F$ is $G$--essential in $F$.
	\end{enumerate}
\end{defn}

\noindent If $G$ acts cocompactly on $F$, then a hyperplane is $G$--essential if and only if it is essential in $F$.

\begin{prop}\label{prop:CScore} Let $\cuco X$ be a proper CAT(0) cube complex admitting a proper cocompact action by a group $\Gamma$, let $F\subseteq \cuco X$ be a convex subcomplex, and let $G\le \Gamma$ act on $F$ cocompactly. Then:
\begin{enumerate}
	\item[(i)] there exists a $G$--invariant convex subcomplex 
$\widehat{F}_G$, called the $G$--essential core of $F$, crossed by 
every essential hyperplane in $F$, on which $G$ acts essentially and 
cocompactly;
	\item[(ii)] $\widehat{F}_G$ is unbounded if and only if $F$ is unbounded;
	\item[(iii)] if $G'\le\text{Aut}(\cuco X)$ also acts on $F$ cocompactly, then $\widehat{F}_{G'}$ is parallel to $\widehat{F}_G$;
	\item[(iv)] if $G'\le G$ is a finite-index subgroup, we can take $\widehat{F}_{G'}=\widehat{F}_G$.
	\item[(v)] the subcomplex $\widehat{F}_G$ is finite Hausdorff distance from $F$.
	\end{enumerate}
\end{prop}

\begin{proof}
By \cite[Proposition 3.5]{CapraceSageev:rank}, $F$ contains a $G$--invariant convex subcomplex $\widehat F_G$ on which $G$ acts essentially and cocompactly (in particular, $\widehat F_G$ is unbounded if and only if $F$ is, and $\dist_{Haus}(F,\widehat F_G)<\infty$).  The hyperplanes of $F$ crossing $\widehat F_G$ are precisely the $G$--essential hyperplanes.  Observe that if $H$ is a hyperplane crossing $F$ essentially, then cocompactness of the $G$--action on $F$ implies that $H$ is $G$--essential and thus crosses $\widehat F_G$.  It follows that if $G,G'$ both act on $F$ cocompactly, then a hyperplane crossing $F$ is $G$--essential if and only if it is $G'$--essential, so $\widehat F_G,\widehat F_{G'}$ cross the same hyperplanes, i.e. they are parallel.  If $G'\le G$, then $\widehat F_G$ is $G'$--invariant, and if $[G:G']<\infty$, then $G'$ also acts cocompactly on $F$, so we can take $\widehat F_G=\widehat F_{G'}$.
\end{proof}

\subsection{Hyperclosure and factor systems}\label{subsec:hyperclosure}

\begin{defn}[Factor system, hyperclosure]\label{defn:factor_system}
The \emph{hyperclosure} of $\cuco X$ is the intersection $\mathfrak F$ of all sets $\mathfrak F'$ of convex subcomplexes of $\cuco X$ that satisfy the following three properties:
\begin{enumerate}
 \item \label{item:hyper_1} $\cuco X\in\mathfrak F'$, and for all combinatorial hyperplanes $H$ of $\cuco X$, we have $H\in\mathfrak F'$;
 \item\label{item:hyper_2} if $F,F'\in\mathfrak F'$, then $\gate_F(F')\in\mathfrak F'$;
 \item\label{item:hyper_3} if $F\in\mathfrak F'$ and $F'$ is parallel to $F$, then $F'\in\mathfrak F'$.
\end{enumerate}
Note that $\mathfrak F$ is $\Aut(\cuco X)$--invariant.  If there exists $\xi$ such that for all $x\in\cuco X$, there are at most $\xi$ elements $F\in\mathfrak F$ with $x\in F$, then, following~\cite{BHS:HHS_I}, we call $\mathfrak F$ a \emph{factor system} for $\cuco X$.
\end{defn}

\begin{rem}\label{rem:fs}
The definition of a factor system in~\cite{BHS:HHS_I} is more general than the 
definition given above.  The assertion that $\cuco X$ has a factor system in the 
sense of~\cite{BHS:HHS_I} is equivalent to the assertion that the hyperclosure 
of $\cuco X$ has finite multiplicity, because any factor system (in the sense 
of~\cite{BHS:HHS_I}) contains all elements of $\mathfrak F$ whose diameters 
exceed a given fixed threshold.  Each of the five conditions in Definition 8.1 
of~\cite{BHS:HHS_I} is satisfied by $\mathfrak F$, except possibly the finite 
multiplicity condition, Definition 8.1.(3).  Indeed, parts (1),(2),(4) of that 
definition are included in Definition~\ref{defn:factor_system} above.  Part (5) 
asserts that there is a constant $p$ so that $\gate_{F}(F')$ is in the factor 
system provided $F,F'$ are and $\diam(\gate_F(F'))\geq p$.  Hence 
Definition~\ref{defn:factor_system}.\eqref{item:hyper_2} implies that this 
condition is satisfied by $\mathfrak F$, with $p=0$.  
\end{rem}

\section{Analysis of the hyperclosure}\label{sec:decomposing}
Fix a proper $\cuco X$ with a group $G$ acting properly and cocompactly.  Let $\mathfrak F$ be the hyperclosure.

\subsection{Decomposition}\label{subsec:decomp}
Let $\mathfrak F_0=\{\cuco X\}$ and, for each $n\geq1$, let $\mathfrak F_n$ be the subset of $\mathfrak F$ consisting of those subcomplexes that can be written in the form $\gate_H(F)$, where $F\in\mathfrak F_{n-1}$ and $H$ is a combinatorial hyperplane.  Hence $\mathfrak F_1$ is the set of combinatorial hyperplanes in $\cuco X$.

\begin{lem}[Decomposing $\mathfrak F$]\label{lem:breakup}
Each $F\in\mathfrak F-\{\cuco X\}$ is parallel to a subcomplex of the form $$\gate_{H_1}(\gate_{H_2}(\cdots\gate_{H_{n-1}}(H_n)\cdots))$$ for some $n\geq1$, where each $H_i$ is a combinatorial hyperplane, i.e. $\mathfrak F/_\parallel=\left(\cup_{n\geq1}\mathfrak F_n\right)/_\parallel$.
\end{lem}

\begin{proof}
This follows by induction, Lemma~\ref{lem:project_curry}, and the definition of $\mathfrak F$.
\end{proof}

\begin{cor}\label{cor:decompose}
$\mathfrak F=\cup_{n\geq0}\mathfrak F_n$.
\end{cor}

\begin{proof}
It suffices to show $\mathfrak F\subseteq\cup_{n\geq0}\mathfrak F_n$.  Let $F\in\mathfrak F$. If $F=\cuco X$, then $F\in\mathfrak F_0$.  Otherwise, by Lemma~\ref{lem:breakup}, there exists $n\geq 1$, a combinatorial hyperplane $H$, and a convex subcomplex $F'\in\cup_{k\leq n}\mathfrak F_k$ with $F\parallel\gate_H(F')$.  Consider $\phi_F(P_F\cong F\times\orth F)$ and choose $f\in\orth F$ so that $\phi_F(F\times\{f\})$ coincides with $F$.  Then $\phi_F(F\times \{f\})$ lies in some combinatorial hyperplane $H'$ -- either $H'=H$ and $F=\gate_H(F)$, or $F$ is non-unique in its parallelism class, so lies in a combinatorial hyperplane in the carrier of a hyperplane crossing $\orth F$.  Consider $\gate_{H'}(\gate_H(F'))$.  On one hand, $\gate_{H'}(\gate_H(F'))\in\cup_{k\leq n+1}\mathfrak F_k$.  But $\gate_{H'}(\gate_H(F'))=F$.  Hence $F\in\cup_{n\geq1}\mathfrak F_n$.
\end{proof}

\subsection{Stabilizers act cocompactly}\label{subsec:cocompact}
The goal of this subsection is to prove that $\stabilizer_G(F)$ acts cocompactly on $F$ for each $F\in\mathfrak F$.  The following lemma is standard but we include a proof in the interest of a self-contained exposition.

\begin{lem}[Coboundedness from finite multiplicity]\label{lem:134} Let $X$ be a metric space and let $G\to\Isom(X)$  act cocompactly, and let
$\mathcal Y$ be a $G$--invariant collection of
subspaces such that every ball intersects finitely many elements of $\mathcal Y$. Then $\stabilizer_G(P)$ acts coboundedly on $P$ for every
$P\in\mathcal Y$.\end{lem}

\begin{proof}
Let $P\in\mathcal Y$, choose a basepoint $r\in X$, and use
cocompactness to choose $t<\infty$ so that $\dist(x,G\cdot r)\leq t$
for all $x\in X$.  Choose $g_1,\ldots,g_{s}\in G$ so that the
$G$--translates of $P$ intersecting $\neb_{10t}(r)$ are exactly
$g_iP$ for $i\leq s$. Since $\mathcal Y$ is $G$--invariant and
locally finite, $s<\infty$.   (In other words, the assumptions guarantee that there are finitely many cosets of 
$\stabilizer_G(P)$ whose corresponding translates of $P$ intersect $\neb_{10t}(r)$, and we have fixed a representative of 
each of these cosets.)  Let $K_r=\max_{i\leq s}\dist(r,g_ir)$.
For each $g\in G$, the translates of $P$ that lie within distance
$10t$ of $g\cdot r$ are precisely $gg_1P,\ldots,gg_sP$.  Letting $K_{gr}=\max_{i\leq s}\dist(gr,gg_ir)$, we have 
$K_{gr}=K_r$ just because $\dist(r,g_i\cdot r)=\dist(g\cdot r,gg_i\cdot
r)$.

Fix a basepoint $p\in P$ and let $q\in P$ be an arbitrary point;
choose $h_p,h_q\in G$ so that $\dist(h_p\cdot r,p)\leq
t,\dist(h_q\cdot r,q)\leq t$.  Without loss of generality, we may
assume that $h_q=1$. Then $\{h_pg_iP\}_{i=1}^s$ is the set of
$P$--translates intersecting $\neb_{10t}(h_p\cdot r)$.  Now, $p\in
P$ and $\dist(h_p\cdot r,p)<10t$, so there exists $i$ so that
$h_pg_iP=P$, i.e. $h_pg_i\in\stabilizer_G(P)$. Finally,
$$\dist(h_pg_i\cdot q,p)\leq\dist(h_pg_i\cdot r,h_pg_i\cdot q)+\dist(h_p\cdot r,p)+\dist(h_pg_i\cdot r,h_p\cdot 
r)\leq2t+K_{h_pr}=2t+K_r,$$
which is uniformly bounded.  Hence the action of $\stabilizer_G(P)$
on $P$ is cobounded.
\end{proof}

\begin{rem}We use Lemma~\ref{lem:134} when $X$ and $P$ are proper, to get a cocompact action.\end{rem}

\begin{lem}\label{lem:cocompact_parallel} Let $\cuco X$ be a proper CAT(0) cube
complex with a group $G$ acting cocompactly. Let $Y, Y'\subset\cuco X$ be parallel convex subcomplexes, then
$\stabilizer_G(Y)$ and $\stabilizer_G(Y')$ are commensurable. Thus,
if $\stabilizer_G(Y)$ acts cocompactly on $Y$, then
$\stabilizer_G(Y)\cap\stabilizer_G(Y')$ acts cocompactly on
$Y'$.\end{lem}

\begin{proof}
Let $T$ be the set of $\stabilizer_G(Y)$--translates of $Y'$.  Then each $gY'\in T$ is parallel to $Y$, and $\dist_{\cuco X}(gY',Y)=\dist_{\cuco X}(Y',Y)$.  Since $\orth Y$ is locally finite, $|T|<\infty$.  Hence $K=\kernel(\stabilizer_G(Y)\to\symmetric(T))$ has finite index in $\stabilizer_G(Y)$ but lies in $\stabilizer_G(Y)\cap\stabilizer_G(Y')$.  By Lemma~\ref{lem:parallel_product}, $K$ acts cocompactly on $Hull(Y\cup Y')$, stabilizing $Y'$, and thus acts cocompactly on $Y'$.
\end{proof}

\begin{defn}\label{defn:FNHK}
Let $H\in\mathfrak F_1$.  For $n\geq1,k\geq0$, let $\mathfrak F_{n,H,k}$ be the set of $F\in\mathfrak F_n$ so that $F=\gate_H(F')$ for some $F'\in\mathfrak F_{n-1}$ with $\dist(H,F')\leq k$. Let $\mathfrak F_{n,H}=\cup_{k\geq0}\mathfrak F_{n,H,k}$ and $\mathfrak F_{n,k}=\cup_{H\in\mathfrak F_1}\mathfrak F_{n,H,k}$.
\end{defn}

\begin{prop}[Cocompactness]\label{prop:f_n_properties}
Let $n\geq 1$. Then, for any $F\in\mathfrak F_n$, $\stabilizer_G(F)$
acts cocompactly on $F$.  Hence $\stabilizer_G(F)$ acts cocompactly on $F$ for each $F\in\mathfrak F$.
\end{prop}

\begin{proof}
The second assertion follows from the first and Corollary~\ref{cor:decompose}.  We argue by double induction on $n,k$ to prove the first assertion, with $k$ as in Definition~\ref{defn:FNHK}.  First, observe that $\mathfrak F_n$, $\mathfrak F_{n,k}$ are $G$--invariant for all
$n,k$. Similarly, $\mathfrak F_{n,H,k}$ is
$\stabilizer_G(H)$--invariant for all $H\in\mathfrak F_1$.

\textbf{Base Case: $n=1$.} From local
finiteness of $\cuco X$, cocompactness of the action of $G$ and
Lemma~\ref{lem:134}, we see that $\stabilizer_G(H)$ acts cocompactly on $H$ for each $H\in\mathfrak F_1$.

\textbf{Inductive Step 1: $(n,k)$ for all $k$ implies $(n+1,0)$.}
Let $F\in\mathfrak F_{n+1,0}$.  Then $F=H\cap F'$, where $H\in\mathfrak F_1$ and 
$F'\in\mathfrak F_n$.  By definition, $F'=\gate_{H'}(F'')$ for some 
$F''\in\mathfrak F_{n-1}$ and $H'\in\mathfrak F_1$.  Thus $K=\stabilizer_G(F')$ 
acts cocompactly on $F'$ by induction.

Let $\mathcal S=\{k(H\cap F'):k\in K\}$, which is a $K$--invariant set of convex subcomplexes of $F'$.  Moreover, since the set of all $K$--translates of $H$ is a locally finite collection, because $\cuco X$ is locally finite and $H$ is a combinatorial hyperplane, $\mathcal S$ has the property that every ball in $F'$ intersects finitely many elements of $\mathcal S$. Lemma~\ref{lem:134}, applied to the cocompact action of $K$ on $F'$, shows that $\stabilizer_K(H\cap F')$ (which equals $\stabilizer_K(F)$), and hence $\stabilizer_G(F)$, acts cocompactly on $F$.

\textbf{Inductive Step 2: $(n,k)$ implies $(n,k+1)$.}  Let $F\in\mathfrak F_{n,k+1}$ so that $F=\gate_H(F')$  with
$H\in\mathfrak F_1$, $F'\in\mathfrak F_{n-1}$ and $d=\dist(H,F')\le
k+1$. If $d\le k$, induction applies. Thus, we can assume that
$d=k+1$. Then there is a product region $F\times[0,d]\subset \cuco
X$ with $F\times\{0\}=F$, and $F\times\{d\}\subset F'$.  Now, $F_1:=F\times\{1\}$ is a parallel copy of $F$ contained in the 
carrier of the hyperplane $H''$ dual to the edge $[0,1]$ of $[0,d]$.  Letting $H'$ be the combinatorial hyperplane parallel 
to $H''$ in $\neb(H'')$ and separated from $F$ by $H''$, we have  $F_1\subset\gate_{H'}(F')$.  Moreover, $d(H',F')\leq 
d-1=k$.  By induction
$L=\stabilizer_G(\gate_{H'}(F'))$ acts cocompactly on
$\gate_{H'}(F')$.

We claim that $F_1=\gate_{H'}(F')\cap\gate_{H'}(H)$. To see this,
note that the hyperplanes that cross $F_1$ are exactly the
hyperplanes that cross $F'$ and $H$. However, those are the
hyperplanes which cross $H'$ and $F'$ which also cross $H$. It easily follows that
the two subcomplexes are equal.

Now let $\mathcal T$ be the set of $L$--translates of
$F_1=\gate_{H'}(F')\cap\gate_{H'}(H)$ in $\gate_{H'}(F')$.  This is
an $L$--invariant collection of convex subcomplexes of
$\gate_{H'}(F')$.  Moreover, each ball in $\gate_{H'}(F')$
intersects finitely many elements of $\mathcal T$. Indeed, $\mathcal
T$ is a collection of subcomplexes of the form $T_\ell=\gate_{\ell
H'}(\ell H)\cap\gate_{H'}(F')$, where $\ell\in L$.  Recall that
$\dist_{\cuco X}(H,H')=1$.  Hence, fixing $y\in\gate_{H'}(F')$ and
$t\geq0$, if $\{T_{\ell_i}\}_{i\in I}\subseteq\mathcal T$ is a
collection of elements of $\mathcal T$, all of which intersect
$\neb_t(y)$, then $\{\ell_iH,\ell_iH'\}_{i\in I}$ all intersect
$\neb_{t+1}(y)$. However, by local finiteness of $\cuco X$ there are
only finitely many distinct $\ell_iH, \ell_i H'$. Further, if
$\ell_i H=\ell_j H$ and $\ell_i H'=\ell_j H'$, then
$T_{\ell_i}=T_{\ell_j}$. Thus, the index set $I$ must be finite.
Hence, by Lemma~\ref{lem:134} and cocompactness of the action of $L$
on $\gate_{H'}(F')$, we see (as in Inductive Step 1) that
$\stabilizer_G(F_1)$ acts cocompactly on $F_1$.  Now, since $F_1$ is
parallel to $F$, we see by Lemma~\ref{lem:cocompact_parallel} that
$\stabilizer_G(F)$ acts cocompactly on $F$.
\end{proof}

The next Lemma explains how to turn the algebraic conditions on the $G$--action described in Section~\ref{sec:uniform} into geometric properties of the convex subcomplexes in $\mathfrak F$. This is of independent interest, giving a complete algebraic characterization of when two cocompact subcomplexes have parallel essential cores.

\begin{lem}[Characterization of commensurable stabilizers]\label{lem:comm_stab} Let $Y_1$ and $Y_2$ be two convex subcomplexes of $\cuco X$ and let $G_i=\stabilizer_G(Y_i)$. Suppose further that $G_i$ acts on $Y_i$ cocompactly. Then $G_1$ and $G_2$ are commensurable if and only of the $G_1$--essential core $\widehat{Y_1}$ and the $G_2$--essential core $\widehat{Y_2}$ are parallel.\end{lem}

\begin{proof} First, if $\widehat{Y_1},\widehat{Y_2}$ are parallel, then 
Lemma~\ref{lem:cocompact_parallel} shows that $\stabilizer_G(\widehat{Y_1}), 
\stabilizer_G(\widehat{Y_2})$ contain $\stabilizer_G(\widehat{Y_1})\cap 
\stabilizer_G(\widehat{Y_2})$ as a finite-index subgroup.  Since 
$\stabilizer_G(\widehat{Y_i})$ contains $G_i$ as a finite-index subgroup, it 
follows that $G_1\cap G_2$ has finite index in $G_1$ and in $G_2$.  
	
Conversely, suppose that $G_1,G_2$ have a common finite-index subgroup. Thus, 
$G_1\cap G_2$ acts cocompactly on both $Y_1$ and $Y_2$.  This implies that 
$Y_1,Y_2$ lie at finite Hausdorff distance, since choosing $r>0$ and $y_i\in 
Y_i$ so that $(G_1\cap G_2)B_r(y_i)=Y_i$, we see that $Y_1$ is in the $d(y_1, 
y_2)+r$ neighbourhood of $Y_2$, and vice-versa. Further, this implies that 
$\widehat{Y_1},\widehat{Y_2}$ lie at finite Hausdorff distance, since 
$\widehat{Y_i}$ is finite Hausdorff distance from $Y_i$. 

Suppose that $\widehat{Y_1},\widehat{Y_2}$ are not parallel.  Then, without loss 
of generality, some hyperplane $H$ of $\cuco X$ crosses $\widehat{Y_1}$ but not 
$\widehat{Y_2}$.  Since $G_1$ acts on $\widehat{Y_1}$ essentially and 
cocompactly,~\cite{CapraceSageev:rank} provides a hyperbolic isometry $g\in G_1$ 
of $\widehat{Y_1}$ so that $g\overleftarrow H\subsetneq\overleftarrow H$, where 
$\overleftarrow H$ is the halfspace of $\cuco X$ associated to $H$ and disjoint 
from $Y_2$.  Choosing $n>0$ so that the translation length of $g^n$ exceeds the 
distance from $Y_2$ to the point in which some $g$--axis intersects $H$, we see 
that $H$ cannot separate $g^n\widehat Y_2$ from the axis of $g$. Thus, 
$g^n\widehat F'\cap\overleftarrow H\neq\emptyset$, whence $\langle g\rangle\cap 
G_2=\{1\}$, contradicting that $G_1$ and $G_2$ are commensurable (since $g$ has 
infinite order).  Thus $\widehat Y_1,\widehat Y_2$ are parallel.
\end{proof}

\subsection{Ascending or descending chains}\label{subsec:ascending_chains}
We reduce Theorem~\ref{thmi:main} to a claim about chains in 
$\mathfrak F$.

\begin{lem}[Finding chains]\label{lem:chain_find}
Let $\mathfrak U\subseteq\mathfrak F$ be an infinite subset satisfying 
$\bigcap_{U\in\mathfrak U}U\ni x$ for some $x\in\cuco X$.  Then one of the 
following holds:
\begin{itemize}
     \item there exists a sequence $\{F_i\}_{i\geq1}$ in $\mathfrak F$ so 
that $x\in F_i\subsetneq F_{i+1}$ for all $i$;
\item there exists a sequence $\{F_i\}_{i\geq1}$ in $\mathfrak F$ so 
that $x\in F_i$ and $F_i\supsetneq F_{i+1}$ for all $i$.
\end{itemize}
\end{lem}

\begin{proof}
Let $\mathfrak F_x\supseteq\mathfrak U$ be the set of $F\in\mathfrak F$ with 
$x\in F$.  Let $\Omega$ be the directed graph with vertex set $\mathfrak F_x$, 
with $(F,F')$ a directed edge if $F\subsetneq F'$ and there does not exist 
$F''\in\mathfrak F_x$ with $F\subsetneq F''\subsetneq F'$.

Let $F_0=\{x\}$.  Since $x$ is the intersection of the finitely many hyperplane 
carriers containing it, $F_0\in\mathfrak F$ and in particular $F_0\in\mathfrak 
F_x$.  Moreover, note that $F_0$ has no incoming $\Omega$--edges, since $F_0$ 
cannot properly contain any other subcomplex.  For any $F\in\mathfrak F_x$, 
either $\Omega$ contains an edge from $F_0$ to $F$, or there exists 
$F'\in\mathfrak F_x$ such that $F_0\subset F'\subset F$.

Hence either $\mathfrak F_x$ contains an infinite ascending or descending 
$\subseteq$--chain, or $\Omega$ is a connected directed graph in which every non-minimal vertex as an immediate predecessor, and every non-maximal vertex has an immediate successor. In the first two 
cases, we are done, so assume that the third holds. In the third case, there is 
a unique vertex namely $F_0$, with no incoming edges and there is a finite 
length directed path from $F_0$ to any vertex.

Let $F\in\mathfrak F_x$ and suppose that $\{F_i\}_i$ is the set of vertices of $\Omega$ so that $(F,F_i)$ is an edge.  For $i\neq j$, we have $F\subseteq F_i\cap F_j\subsetneq F_i$, so since $F_i\cap F_j=\gate_{F_i}(F_j)\in\mathfrak F$, we have $F_i\cap F_j=F$.

The set $\{F_i\}_i$ is invariant under the action of $\stabilizer_G(F)$.  Also, 
by Proposition~\ref{prop:f_n_properties}, $\stabilizer_G(F)$ acts cocompactly on 
$F$.  A $0$--cube $y\in F$ is \emph{diplomatic} if there exists $i$ so that $y$ 
is joined to a vertex of $F_i-F$ by a $1$--cube in $F_i$.  Only uniformly 
finitely many $F_i$ can witness the diplomacy of $y$ since $\cuco X$ is 
uniformly locally finite and $F_i\cap F_j=F$ whenever $i\neq j$.  Also, $y$ is 
diplomatic, witnessed by $F_{i_1},\ldots,F_{i_k}$, if and only if $gy$ is 
diplomatic, witnessed by $gF_{i_1},\ldots,gF_{i_k}$, for each 
$g\in\stabilizer_G(F)$.  Since $\stabilizer_G(F)\curvearrowright F$ cocompactly, 
we thus get $|\{F_i\}_i/\stabilizer_G(F)|<\infty$.


Let $\widehat\Omega$ be the graph with a vertex for each $F\in\mathfrak F$ containing a point of $G\cdot x$ and a directed edge for minimal containment as above.
Then $\widehat\Omega$ is a graded directed graph as above.  For each $n\geq0$, let $\mathcal S_n$ be the set of vertices in $\widehat\Omega$ at distance $n$ from a
minimal element.  The above argument shows that $G$ acts cofinitely on each 
$\mathcal S_n$, and thus $\widehat\Omega/G$ is locally finite.  Hence, by 
K\"{o}nig's infinity lemma, either $\widehat\Omega/G$ contains a directed ray or 
$\widehat\Omega^{(0)}/G$ is finite. In the former case, $\widehat{\Omega}$ must 
contain a directed ray, in which case there exists $\{F_i\}\subseteq \mathfrak 
F$ with $F_i\subsetneq F_{i+1}$ for all $i$. Up to translating by an appropriate 
element of $G$, we can assume that $x\in F_1$. The latter case means that the 
set of $F\in\mathfrak F$ such that $F\cap G\cdot x\neq\emptyset$ is $G$--finite. 
 But since $G$ acts properly and cocompactly on $\cuco X$, any $G$--invariant 
$G$--finite collection of subcomplexes whose stabilizers act cocompactly has 
finite multiplicity, a contradiction.  
\end{proof}

\section{Orthogonal complements of compact sets and the 
hyperclosure}\label{sec:orthocomp}
We now characterise $\mathfrak F$ in a CAT(0) cube 
complex $\cuco X$, without making use of a group action.

\begin{lem}\label{lem:contravariant}
Let $A\subseteq B\subseteq\cuco X$ be convex subcomplexes and let $a\in A$. Then
$\phi_B(\{a\}\times\orth B)\subseteq\phi_A(\{a\}\times\orth A)$.
\end{lem}

\begin{proof}
Let $x\in\phi_B(\{a\}\times\orth B)$.  Then every hyperplane $H$ separating $x$ from $a$ separates two parallel copies of $B$ and thus separates two parallel copies of $A$, since $A\subseteq B$.  It follows from Lemma~\ref{lem:orthocomp} that every hyperplane separating $a$ from $x$ crosses $\phi_A(\{a\}\times\orth A)$, whence $x\in\phi_A(\{a\}\times\orth A)$.
\end{proof}

Given a convex subcomplex $F\subseteq\cuco X$, fix a base $0$--cube $f\in F$ and 
for brevity, let $\orth F=\phi_F(\{f\}\times\orth F)\subseteq\cuco X$.  Note 
that $f\in\orth F$, and so we let $\orth{\orth F}=\phi_{\orth 
F}\left(\{f\}\times\orth{(\orth F)}\right)$ (here, the $\orth{(\orth F)}$ is the 
abstract orthogonal complement of $\orth F$), which again contains $f$, and so 
we can similarly define $\orth{\left(\orth{\left(\orth{F}\right)}\right)}$ etc.

\begin{lem}\label{lem:3_for_1}Let $F$ be a convex subcomplex of $\cuco X$. Then
$\orth{\left(\orth{\left(\orth{F}\right)}\right)}=\orth{F}$.\end{lem}

\begin{proof} If $F$ is a convex subcomplex, there is a parallel copy of $\orth{F}$ based at each $0$--cube of $F$, since $F\times\orth{F}$ is a convex
subcomplex of $\cuco X$. Thus
$F\hookrightarrow_{\parallel}\orth{\left(\orth{F}\right)}$, and by
Lemma~\ref{lem:contravariant} we have
$\orth{F}\supseteq\orth{\left(\orth{\left(\orth{F}\right)}\right)}$.
To obtain the other inclusion, we show that every parallel copy of
$F$ is contained in a parallel copy of
$\orth{\left(\orth{F}\right)}$. This is clear since, letting
$A=\orth{F}$, we have that $\phi_A(A\times \orth A)$ is a convex subcomplex
of $\cuco X$, but $F\subset \orth{A}$ by the above, and thus
$\phi_F(F\times\orth{F})\subseteq
\phi_{\orth F}(\orth{F}\times\orth{\left(\orth{F}\right)})$, both of which are
convex subcomplexes of $\cuco X$. Hence
$\orth{F}\subseteq\orth{\left(\orth{\left(\orth{F}\right)}\right)}$,
completing the proof.
\end{proof}

\subsection{Characterisation of $\mathfrak F$ using orthogonal 
complements of compact sets}\label{sec:orthocompcompact}
In this section, we assume that $\cuco X$ is locally finite, but do not need a 
group action.

\begin{thm}\label{thm:factor_system_compact_orthocomp}
Let $F\subset\cuco X$ be a convex subcomplex.  Then $F\in\mathfrak F$ if and 
only if there exists a compact convex subcomplex $C$ so that $\orth C=F$.  
\end{thm}

\begin{proof}
Let $C$ be a compact convex subcomplex of $\cuco X$.  Let $H_1,\ldots,H_k$ be 
the hyperplanes crossing $C$.  Fix a basepoint $x\in C$, and suppose the $H_i$ 
are labeled so that $x\in \neb(H_i)$ for $1\leq i\leq  m$, and 
$x\not\in\neb(H_i)$ for $i>m$, for some $m\leq k$.  Let 
$F=\bigcap_{i=1}^k\gate_{H_1}(H_i)$, which contains $x$.  Any hyperplane 
$H$ crosses $\phi_C(\{x\}\times\orth C)$ if and only if $H$ crosses each $H_i$, 
which occurs if and only if $H$ crosses $F$.  Hence $F=\phi_C(\{x\}\times\orth 
C)$, as required.

We now prove the converse.  Let $F\in\mathfrak F_n$ for $n\geq 1$.  If $n=1$ and 
$F$ is a combinatorial hyperplane, $F=\orth e$ for some $1$--cube $e$ of $\cuco 
X$.  Next, assume that $n\geq 2$ and write $F=\gate_H(F')$ where $F'\in\mathfrak 
F_{n-1}$ and $H$ is a combinatorial hyperplane.  Induction on $n$ gives 
$F'=\orth {(C')}$ for some compact convex subcomplex $C'$.

Let $e$ be a $1$--cube with orthogonal complement $H\in\mathfrak
F_1$, chosen as close as possible to $C'$, so that
$\dist(e,C')=\dist(H,C')$.  In particular, any hyperplane separating $e$ from 
$C'$ separates $H$ from $C'$.  Moreover, we can and shall assume that $C'$ was 
chosen in its parallelism class so that $\dist(e,C')$ is minimal when $e,C'$ are 
allowed to vary in their parallelism classes.  

Let $C$ be the convex hull of (the
possibly disconnected set) $e\cup C'$. 

We claim that $\gate_H(F')=\{x\}\times \orth C$.  First, suppose that $V$ is a 
hyperplane crossing $\{x\}\times\orth C$.  Then $V$ separates two parallel 
copies of $C$, each of which contains a parallel copy of $e$ and one of $C'$.  
Hence $V$ crosses $H$ and $F'$, so $V$ crosses $\gate_H(F')$.  Thus 
$\{x\}\times\orth C\subseteq\gate_H(F')$.

Conversely, suppose $V$ is a hyperplane crossing $\gate_H(F')$, i.e. crossing 
$H$ and $F'$.  To show that $V$ crosses $\{x\}\times\orth C$, it suffices to 
show that $V$ crosses every hyperplane crossing $C$.  If $W$ crosses $C$, then 
either $W$ separates $e,C'$ or crosses $e\cup C'$.  In the latter case, $V$ 
crosses $W$ since $V$ crosses $H$ and $\orth{(C')}=F'$.  In the former case, 
since $e,C'$ are as close as possible in their parallelism classes, $W$ 
separates $e,C'$ only if it separates $H$ from $C'\times\orth{(C')}$, so $W$ 
must cross $V$.  Hence $\gate_H(F')\subseteq\{x\}\times\orth C$. Since only finitely many hyperplanes $V$ either cross $e$, cross $C'$, 
or separate $e$ from $C'$, the subcomplex $C$ is compact.
\end{proof}

\begin{cor}\label{cor:f_characterise}
If $F\in\mathfrak F$, then $\orth{(\orth F)}=F$.
\end{cor}

\begin{proof}
If $F\in\mathfrak F$, then $F=\orth C$ for some compact $C$, by 
Theorem~\ref{thm:factor_system_compact_orthocomp}, and hence 
$\orth{(\orth{(\orth C)})}=\orth{(\orth F)}=\orth C=F$, by 
Lemma~\ref{lem:3_for_1}.
\end{proof}

\section{Auxiliary conditions}\label{sec:uniform}
In this section, the group $G$ acts geometrically on the proper 
CAT(0) cube complex $\cuco X$.

\subsection{Rotation}
\begin{defn}[Rotational]\label{defn:forgetful}
The action of $G$ on $\cuco X$ is \emph{rotational} if the following 
holds.  For each hyperplane $B$, there is a finite-index subgroup 
$K_B\leq\stabilizer_G(B)$ so that for all hyperplanes $A$ with 
$\dist(A,B)>0$, and all $k\in K_B$, the carriers $\neb(A)$ and 
$\neb(kA)$ are either equal or disjoint.
\end{defn}

\begin{rem}\label{rem:stabilizer_special}
For example, if $G\backslash\cuco X$ is (virtually) special, then $G$ acts 
rotationally on $\cuco X$, but one can easily make examples of non-cospecial 
rotational actions on CAT(0) cube complexes. 
\end{rem}

To illustrate how to apply rotation, we first prove a lemma about $\mathfrak 
F_2$.

\begin{lem}[Uniform cocompactness in 
$\mathfrak F_2$ under rotational actions]\label{lem:f_2_uniformly_cocompact}
Let $G$ act properly, cocompactly, and rotationally on $\cuco X$.  Then for any 
ball $Q$ in $\cuco X$, there exists $s\ge0$, depending only on $\cuco X,$ and 
the radius of $Q$, so that for 
all $A,B\in\mathfrak F_1$, at most $s$ distinct translates of $\gate_B(A)$ can 
intersect $Q$.
\end{lem}

\begin{proof}
Note that if $B,gB$ are in the same $G$--orbit, and $K_B\leq\stabilizer_G(B)$ 
witnesses the rotation of the action at $B$, then $K_B^g$ does the same for 
$gB$, so we can assume that the index of $K_B\in\stabilizer_G(B)$ is uniformly 
bounded by some constant $\iota$ as $B$ varies over the (finitely many orbits 
of) combinatorial hyperplanes.

Next, note that it suffices to prove the claim for $Q$ of radius $0$, since the 
general statement will then follow from uniform properness of $\cuco X$.

Finally, it suffices to fix combinatorial hyperplanes $B$ and $A$ and bound the 
number of $\stabilizer_G(B)$--translates of $A$ whose projections on $B$ 
contain 
some fixed $0$--cube $x\in B$, since only boundedly many translates of $B$ can 
contain $x$.

We can assume that $A$ is disjoint from $B$, for otherwise  $\gate_B(A)=A\cap 
B$, and the number of translates of $A$ containing $x$ 
is bounded in terms of $G$ and $\cuco X$ only.

Now suppose $\dist(A,B)>0$.  
First, let 
$\{g_1,\ldots,g_k\}\subset K_B$ be such that the translates 
$g_i\gate_B(A)$ are all distinct and 
$x\in\bigcap_{i=1}^k\gate_B(g_iA)=\bigcap_{i=1}^kg_i\gate_B(A)$.  For 
simplicity, we can and shall assume that $g_1=1$.

We can also assume, by multiplying our eventual bound by $2$, that the $g_iA$ 
all lie on the same side of $B$, i.e. the hyperplane $B'$ whose carrier is 
bounded by $B$ and a parallel copy of $B$ does not separate any pair of the 
$g_iA$.  By rotation, $A,g_2A$ are disjoint, and hence separated by some 
hyperplane $V$.  

Since $V$ cannot separate $\gate_B(A),\gate_B(g_2A)$, we have that $V$ 
separates either $A$ or $g_2A$ from $B$.  (The other possibility is that 
$V=B$ but we have ruled this out above.)  Up to relabelling, 
we can assume the former.  Then, for $i\ge2$, we have that $g_iV$ separates 
$g_iA$ from $g_iB=B$.  Moreover, by choosing $V$ as close as possible to $B$ 
among hyperplanes that separate $A$ from $B$ and $g_2A$, we see that the 
hyperplanes $\{g_iV\}_{i=1}^k$ have pairwise-intersecting carriers, and at 
least 
two of them are distinct.  This contradicts the rotation hypothesis unless 
$k=1$.

More generally, the above argument shows that if $\{g_1,\ldots,g_k\}\subset 
\stabilizer_G(B)$ are such that the translates $g_i\gate_B(A)$ are all distinct 
and $x\in\bigcap_{i=1}^k\gate_B(g_iA)=\bigcap_{i=1}^kg_i\gate_B(A)$, then the 
number of $g_i$ belonging to any given left coset of $K_B$ in 
$\stabilizer_G(B)$ is uniformly bounded.  Since 
$[\stabilizer_G(B):K_B]\leq\iota$, the lemma follows.
\end{proof}

More generally:

\begin{lem}\label{lem:uniform_cocompactness_in_F}
Let $G$ act properly, cocompactly, and rotationally on $\cuco X$.  Then for all $\rho\geq0$, there exists a constant $s'$ so 
that the following holds.  Let $F\in\mathfrak 
F$.  Then at most $s'$ distinct $G$--translates of $F$ can 
intersect any $\rho$--ball in $\cuco X$.
\end{lem}

\begin{proof}
As in the proof of Lemma~\ref{lem:f_2_uniformly_cocompact}, it suffices to 
bound the number of $G$--translates of $F$ containing a given $0$--cube $x$.  
As in the same proof, it suffices to bound the number of 
$\stabilizer_G(B)$--translates of $F$ containing $x$, where $B$ is a 
combinatorial hyperplane for which $x\in F\subset B$.

By the first paragraph of the proof of Theorem~\ref{thm:factor_system_compact_orthocomp}, there exists $n$ and combinatorial hyperplanes 
$A_1,\ldots,A_n$ such that 
$F=\bigcap_{i=1}^n\gate_B(A_i)$.  If $A_n$ is parallel to 
$B$, then $F=\bigcap_{i=1}^{n-1}\gate_B(A_i)$, so by choosing a smallest such 
collection, we have that no $A_i$ is parallel to $B$.    

Let $g_1,\ldots,g_k\in K_B$ have the property that the $g_iF$ are 
all distinct and contain $x$.  Note that $g_iF=\bigcap_{j=1}^n\gate_B(g_iA_j)$. 
 
Now, if $A_j$ is disjoint from $B$, then rotation implies that for 
all $i$, either $g_iA_j=A_j$ or $g_iA_j\cap A_j=\emptyset$.  If 
$g_iA_j\neq A_j$, there is a hyperplane $V$ separating them, and $V$ cannot 
cross or coincide with $B$ (as in the proof of 
Lemma~\ref{lem:f_2_uniformly_cocompact}).  Hence $V$ separates $A_j$, say, from 
$B$.  So $g_iV$ separates $g_iA_j$ from $B$, and $g_iV\neq g_jV$.  By choosing $V$ 
as close as possible to $B$, we have (again as in 
Lemma~\ref{lem:f_2_uniformly_cocompact}) that $V$ and $g_iV$ cross or osculate, 
which contradicts rotation.  Hence $g_iA_j=A_j$ for all such $i,j$.

Let $J$ be the set of $j\leq n$ so that $A_j$ is disjoint from $B$, so that 
$\bigcap_{j\in J}\gate_B(A_j)$ is fixed by each $g_i$.  Let 
$J'=\{1,2,\ldots,n\}-J$.  Note that for all $j\in J'$, the combinatorial 
hyperplane $A_j$ is one of at most $\chi$ combinatorial hyperplanes that 
contain $x$ and $\chi$ is the maximal degree of a vertex in $\cuco X$.  
Moreover, $\gate_B(A_j)=A_j\cap B$.

 Since each $g_i$ fixes $\bigcap_{j\in J}\gate_B(A_j)$, $k$ must be bounded in 
terms of the number of translates of $\bigcap_{j\in J'}\gate_B(A_j)$ containing 
$x$; since we can assume that the $A_j$ contain $x$, this 
follows.

As in the proof of Lemma~\ref{lem:f_2_uniformly_cocompact}, if 
$g_1,\ldots,g_k\in \stabilizer_G(B)$ have the property that the $g_iF$ are 
all distinct and contain $x$, then the number of $g_i$ belonging to any 
particular $hK_B,h\in\stabilizer_G(B)$ is uniformly bounded, and the number of 
such cosets is bounded by $\iota$, so the number of such 
$\stabilizer_G(B)$--translates of $F$ containing $x$ is uniformly bounded.
\end{proof}

We now prove Theorem~\ref{thmi:main} in the special case where $G$ acts on 
$\cuco X$ rotationally.

\begin{cor}\label{cor:rotation_FS}
Let $G$ act properly, cocompactly, and rotationally on the proper CAT(0) cube 
complex $\cuco X$.  Then $\mathfrak F$ is a factor system.
\end{cor}

\begin{proof}
By Lemma~\ref{lem:uniform_cocompactness_in_F}, there exists $s'<\infty$ so that 
for all $F\in\mathfrak F$, at most $s'$ distinct $G$--translates of $F$ can 
contain a given point.  By uniform properness of $F$ and the proof of 
Lemma~\ref{lem:134}, there exists $R<\infty$ so that each $F\in\mathfrak F$ has 
the following property: fix a basepoint $x\in F$.  Then for any $y\in F$, there 
exists $g\in\stabilizer_G(F)$ so that $\dist(gx,y)\leq R$.  Hence there exists 
$k$ so that for all $F$, the complex $F$ contains at most $k$ 
$\stabilizer_G(F)$--orbits of cubes. 

\textbf{Conclusion in the virtually torsion-free case:}  If $G$ is 
virtually torsion-free, then (passing to a finite-index 
torsion-free subgroup) $G\backslash\cuco X$ is a compact nonpositively-curved 
cube complex admitting a local isometry $\stabilizer_G(F)\backslash F\to 
G\backslash\cuco X$, where $\stabilizer_G(F)\backslash F$ is a 
nonpositively-curved cube complex with at most $k$ cubes.  Since there are only 
finitely many such complexes, and finitely many such local isometries, the 
quotient $G\backslash\mathfrak F$ is finite.  Since each $x\in\mathfrak S$ is 
contained in boundedly many translates of each $F\in\mathfrak F$, and there are 
only finitely many orbits in $\mathfrak F$, it follows that $x$ is contained in 
boundedly many elements of $\mathfrak F$, as required.

\textbf{General case:}  Even if $G$ is not virtually torsion-free, we can argue 
essentially as above, except we have to work with nonpositively-curved 
orbi-complexes instead of nonpositively-curved cube complexes.

First, let $\cuco Y$ be the first barycentric subdivision of $\cuco X$, so that 
$G$ acts properly and cocompactly on $\cuco Y$ and, for each cell $y$ of $\cuco 
Y$, we have that $\stabilizer_G(y)$ fixes $y$ pointwise (see~\cite[Chapter 
III.$\mathcal C$.2]{Bridson:NPC}.)  Letting $F'$ be the first barycentric 
subdivision of $F$, we see that $F'$ is a subcomplex of $\cuco Y$ with the same 
properties with respect to the $\stabilizer_G(F)$--action.  Moreover, $F'$ has 
at most $k'$ $\stabilizer_G(F)$--orbits of cells, where $k'$ depends on 
$\dimension\cuco X$ and $k$, but not on $F$.

The quotient $G\backslash\cuco Y$ is a complex of groups whose cells are 
labelled by finitely many different finite subgroups, and the same is true for 
$\stabilizer_G(F)\backslash F$.  Moreover, we have a morphism of complexes of 
groups $\stabilizer_G(F)\backslash F\to G\backslash\cuco Y$ which is injective 
on local groups.  Since $G$ acts on $\cuco X$ properly, the local groups in 
$G\backslash\cuco Y$ are finite.  Hence there are boundedly many cells in 
$\stabilizer_G(F)\backslash F$, each of which has boundedly many possible local 
groups (namely, the various subgroups of the local groups for the cells of 
$G\backslash\cuco Y$).  Hence there are finitely many choices of 
$\stabilizer_G(F)\backslash F$, and thus finitely many $G$--orbits in $\mathfrak 
F$, and we can conclude as above.
%
%
%
\end{proof}

\subsection{Weak finite height and essential index conditions}

\begin{defn}[Weak finite height condition]\label{defn:weak_height}
Let $G$ be a group and $H\leq G$ a subgroup.
The subgroup $H$ satisfies the \emph{weak finite height condition} if the following 
holds.  Let $\{g_i\}_{i\in I}$ be an infinite subset of $G$ so that 
$H\cap\bigcap_{i\in J}H^{g_i}$ is infinite whenever $J\subset I$ is a finite 
subset. 
 Then there exists $i,j$ so that $H\cap H^{g_{i}}=H\cap H^{g_j}$.
\end{defn}

\begin{defn}[Noetherian Intersection of Conjugates Condition 
(NICC)]\label{defn:weak_height}
Let $G$ be a group and $H\leq G$ a subgroup.
The subgroup $H$ satisfies the \emph{Noetherian intersection of conjugates 
condition} (NICC) if the following 
holds.  Let $\{g_i\}_{i=1}^{\infty}$ be an infinite subset of distinct elements 
of  $G$ so that 
$H_n=H\cap\bigcap_{i=1}^nH^{g_i}$ is infinite for all $n$, then there exists $\ell>0$ so that for all $j,k\ge \ell$, $H_j$ and $H_k$ are commensurable.

\end{defn}

\begin{defn}[Conditions for 
hyperplanes]\label{defn:weakheight_hyp}
Let $G$ act on the CAT(0) cube complex $\cuco X$.  Then the action satisfies 
the \emph{weak finite height condition for hyperplanes} or respectively \emph{NICC for 
hyperplanes} if, for each hyperplane 
$B$ of 
$\cuco X$, the subgroup $\stabilizer_G(B)\leq G$ satisfies the weak finite height 
condition, or NICC, respectively.
\end{defn}

\begin{rem}\label{rem:finite_height}
Recall that $H\leq G$ has \emph{finite height} if there exists $n$ so that any 
collection of at least $n+1$ distinct left cosets of $H$ has the property that 
the intersection of the corresponding conjugates of $H$ is finite.  Observe 
that if $H$ has finite height, then it satisfies both the weak finite height condition 
and NICC, but 
that the converse does not hold.
\end{rem}

\begin{defn}[Essential index condition]\label{defn:essential_index_condition}
The action of $G$ on $\cuco X$ satisfies the \emph{essential index condition} 
if there exists $\zeta\in\naturals$ so that for all $F\in\mathfrak F$ we have 
$[\stabilizer_G(\widehat{F}):\stabilizer_G(F)]\leq\zeta$, 
where $\widehat{F}$ is the $\stabilizer_G(F)$--essential core 
of $F$.
\end{defn}

\subsection{Some examples where the auxiliary conditions are 
satisfied}\label{subsec:examples}
We now briefly consider some examples illustrating the various hypotheses.  Our 
goal here is just to illustrate the conditions in simple cases. 

\subsubsection{Special groups}
Stabilizers of hyperplanes in a right-angled Artin groups are simply 
special subgroups generated by the links of vertices. Let $\Gamma$ be a graph 
generating a right angled Artin group $A_\Gamma$ and let $\Lambda$ be any 
inducted subgraph. Then $A_\Lambda$ is a special subgroup of $A_\Gamma$, and 
$A_{\Lambda}^{g_i}$ has non-trivial intersection with $A_{\Lambda}^{g_j}$ if 
and 
only if $g_ig_j^{-1}$ commutes with some subgraph $\Lambda'$ of $\Lambda$. 
Further, their intersection is conjugate to the special subgroup 
$A_{\Lambda'}$. 
The weak finite height condition and NICC follow.  The essential index condition also 
holds since each $A_{\Lambda}$ acts essentially on the corresponding element of 
$\mathfrak F$, which is just a copy of the universal cover of the Salvetti 
complex of $A_{\Lambda}$.  In fact, these considerations show that hyperplane 
stabilisers in RAAGs have finite height.  It is easily verified that these 
properties are inherited by 
subgroups arising from compact local isometries to the Salvetti complex, 
reconfirming that (virtually) compact special cube complexes have factor 
systems in their universal covers.

\subsubsection{Non-virtually special lattices in products of trees}\label{subsubsec:bad}
The uniform lattices in products of trees 
from~\cite{WiseThesis,BurgerMozes,Rattagi,JanzenWise} do not 
satisfy the weak finite height condition, but they do satisfy NICC and the essential 
index condition.  

Indeed, let $G$ be a cocompact lattice in $\Aut(T_1\times T_2)$, where 
$T_1,T_2$ are locally finite trees.  If $A,B$ are 
disjoint hyperplanes, then $\gate_B(A)$ is a parallel copy of some $T_i$, i.e. 
$\gate_B(A)$ is again a hyperplane; otherwise, if $A,B$ cross, then 
$\gate_B(A)$ is a single point.  The essential index condition follows 
immediately, as does the NICC.  However, $G$ can be chosen so that there are 
pairs of parallel hyperplanes $A,B$ so that 
$\stabilizer_G(A)\cap\stabilizer_G(B)$ has arbitrarily large (finite) index in 
$\stabilizer_G(B)$, so the weak finite height condition fails.


\subsubsection{Graphs of groups}
Let $\Gamma$ be a finite graph of groups, where each vertex group $G_v$ acts 
properly and cocompactly on a CAT(0) cube complex $\cuco X_v$ with a factor 
system $\mathfrak F_v$, and each edge group $G_e$ acts properly and cocompactly 
on a CAT(0) cube complex $\cuco X_e$, with a factor system $\mathfrak F_e$, so 
that the following conditions are satisfied, where $v,w$ are the vertices of 
$e$:
\begin{itemize}
     \item there are $G$--equivariant convex embeddings $\cuco X_e\to\cuco 
X_v,\cuco X_w$
\item these embedding induce injective maps $\mathfrak F_e\to\mathfrak 
F_v,\mathfrak F_w$.
\end{itemize}

\noindent If the action of $G$ on the Bass-Serre tree is acylindrical, then one 
can argue essentially as in the proof of~\cite[Theorem 8.6]{BHS:HHS_II} to 
prove that the resulting tree of CAT(0) cube complexes has a factor system.  
Moreover, ongoing work on improving~\cite[Theorem 8.6]{BHS:HHS_II} indicates 
that one can probably obtain the same conclusion in this setting without this 
acylindricity hypothesis.

Of course, one can imagine gluing along convex cocompact subcomplexes that 
don't belong to the factor systems of the incident vertex 
groups.  Also, we believe that the property of being cocompactly cubulated with 
a factor system is preserved by taking graph products, and that one can prove this by induction on the size of the graph by 
splitting along link subgroups.  This is 
the subject of recent work in the hierarchically hyperbolic 
setting; see~\cite{BerlaiRobbio}.

\subsubsection{Cubical small-cancellation quotients}
There are various ways of building more exotic examples of non-virtually 
special cocompactly cubulated groups using groups.  
In~\cite{JankiewiczWise}, Jankiewicz-Wise construct a group $G$ that 
is cocompactly cubulated but does not virtually split.  They start with a group 
$G'$ of the type discussed in~Remark~\ref{rem:stabilizer_special} and consider 
a small-cancellation quotient of the free product of several copies of $G'$.  
This turns out to satisfy strong cubical small-cancellation conditions 
sufficient to produce a proper, cocompact action of $G$ on a CAT(0) cube 
complex.  However, it appears that the small-cancellation conditions needed to 
achieve this are also strong enough to ensure that the NICC and essential index 
properties pass from $G'$ to $G$.  The key points are that $G$ is hyperbolic 
relative to $G'$, and each wall in $G$ intersects each coset of $G'$ in at most 
a single wall (Lemma 4.2 and Corollary 4.5 of~\cite{JankiewiczWise}). 

\subsubsection{A non-rotational example}\label{subsubsec:non_rotational}
Let $Y$ be a compact nonpositively-curved cube complex whose fundamental group $G$ has the following two properties:
\begin{itemize}
     \item $G$ has no proper finite-index subgroup;
     \item there exists $g\in G$ such that $g$ is represented by a based combinatorial path $L\to Y$ that is a local 
isometry.
\end{itemize}
The examples mentioned in Subsection~\ref{subsubsec:bad} show that we can take $Y$ to have universal cover the product of 
two trees.

Let $M'$ be a copy of $[0,2]\times[0,|L|]$, endowed with the product cubical structure in which $[0,2]$ and $[0,|L|]$ are 
regarded as cube complexes with $0$--cubes at integer points.  Let $M$ be the quotient of $M'$ obtained by identifying 
$[0,2]\times\{0\}$ and $[0,2]\times\{|L|\}$ by an orientation-reversing combinatorial isometry, so that $M$ is a M\"{o}bius 
strip tiled by squares.  Form a cube complex $X$ from $Y\times[0,1]\sqcup M$ by identifying $L$ with $\{1\}\times[0,|L|]$ 
(here we think of $L$ as a the path $L\to Y\times\{0\}\hookrightarrow Y\times[0,1]$).  Then $X$ is nonpositively curved, 
because $L\to Y\times[0,1]$ is a local isometry and  $L\to M$ is a local isometry.  

Let $\widetilde X$ be the universal cover of $X$, on which $\pi_X=G$ acts freely and cocompactly.  Now, the preimage of 
$Y\times\{\frac12\}$ under the covering map $\widetilde X\to X$ has a single component, which is a hyperplane that we call 
$B$.  By construction, the stabiliser of $B$ is $G$, which has no proper finite-index subgroups.  Now, let $A$ be a 
hyperplane of $\widetilde X$ projecting to an immersed hyperplane of $M$ dual to the image of $[0,1]\times\{0\}$.  Then 
$\neb(A)\cap\neb(gA)$ intersect (along an elevation of $L$), while $A$ and $B$ do not cross.  Hence the action of $G$ on 
$\widetilde X$ is non-rotational; because $G=\stabilizer_G(B)$ has no proper finite-index subgroup, it was sufficient to 
find a single hyperplane $A$ and a single $g\in G$ violating the condition in Definition~\ref{defn:forgetful}.

For a less self-contained example, it appears that $G$ can be chosen so that the action of $G$ on $Y$ is not rotational.  In 
fact~\cite{BurgerMozes} and~\cite{WiseThesis} contain examples where $G$ is acting geometrically on $T_1\times T_2$, with 
$T_1,T_2$ trees, but the induced actions on the two factors are not discrete.  In particular, it seems that for each edge 
$e$ of $T_1$, and any $r\geq0$, there is a vertex of $T_1$ at distance $r$ from $e$ so that the stabiliser of $e$ 
nontrivially permutes the edges incident to that vertex, and this seems to be an obstruction to the action being rotational. 
 (This is in stark contrast to the case where $G$ is virtually a product of free groups, in which case the action is 
rotational.)

\section{$\mathfrak F$ is closed under orthogonal complementation, given a 
group action}\label{subsec:closed_under_complementation}
We now assume that $\cuco X$ is a locally finite CAT(0) cube complex on which 
the group $G$ acts properly and cocompactly.  Let $\mathfrak F$ be the 
hyperclosure in $\cuco X$ and let $B$ be a constant so that each $0$--cube $x$ 
of $\cuco X$ lies in $\leq B$ combinatorial hyperplanes.

For convex subcomplexes $D,F$ of $\cuco X$, we write $F=\orth D$ to mean 
$F=\phi_D(\{f\}\times\orth D)$ for some $f\in D$, though we may abuse notation, 
suppress the $\phi_D$, and write e.g. $\{f\}\times \orth D$ to mean 
$\phi_D(\{f\}\times \orth D)$ when we care about the specific point $f$.

\begin{prop}[$\mathfrak F$ is closed 
under orthogonal complements]\label{prop:closed_uncer_complementation}
Let $G$ act on $\cuco X$ properly and cocompactly.  Suppose that \textbf{one of the following holds}:
\begin{itemize}
 \item the $G$--action on $\cuco X$ satisfies the weak finite height property for hyperplanes;
 \item the $G$--action on $\cuco X$ satisfies the essential index condition and the NICC for hyperplanes;
 \item $\mathfrak F$ is a factor system.
\end{itemize}

Let $A$ be a convex 
subcomplex of $\cuco X$.  Then $\orth 
A\in\mathfrak F$.  Hence $\stabilizer_G(\orth A)$ acts on $\orth A$ cocompactly.

In particular, for all $F\in\mathfrak F$, we have that $\orth F\in\mathfrak F$.
\end{prop}

We first need a lemma.

\begin{lem}\label{lem:proj_to_y}
Let $A\subset\cuco X$ be a convex subcomplex with $\diam(A)>0$ and let $x\in A^{(0)}$.  Let $H_1,\ldots,H_k$ 
be all of the hyperplanes intersecting $A$ whose carriers contain $x$, so that 
for each $i$, there is a combinatorial hyperplanes $H_i^+$ associated to $H_i$ 
with $x\in H_i^+$.  Let $Y=\cap_{i=1}^kH^+_i$. Let $\mathcal S$ be the set of all combinatorial 
hyperplanes associated to hyperplanes crossing $A$.  Then $$\orth A=\bigcap_{H'\in\mathcal S}\gate_{Y}(H'),$$ where $\orth 
A$ denotes the orthogonal complement of $A$ at $x$.  If $A$ is unique in its parallelism class, then $\orth A=\{x\}$.  
Finally, if $\diam(A)=0$, then $\orth A=\cuco X$.
\end{lem}

\begin{proof}
If $A$ is a single $0$--cube, then $\orth A=\cuco X$ by definition.  Hence suppose that $\diam(A)>0$.

Let $H'\in\mathcal S$.  
Since $\gate_Y(H'\cap A)\subseteq Y\cap A=\{x\}$, we see that $x\in\gate_Y(H')$.  Suppose that $y\in\orth A$.  Then every hyperplane $V$ separating 
$y$ from $x$ crosses each of the hyperplanes $H'$ crossing $A$, and thus crosses 
$Y$, whence $y\in\gate_Y(H')$ for each $H'\in\mathcal S$.  Thus 
$\orth{A}\subseteq\bigcap_{H'\in\mathcal S}\gate_{Y}(H')$.  On the other hand, 
suppose that $y\in\bigcap_{H'\in \mathcal S}\gate_{Y}(H')$.  Then every 
hyperplane $H'$ separating $x$ from $y$ crosses every hyperplane crossing $A$, 
so $y\in\orth A$.  This completes the proof that $\orth A=\bigcap_{H'\in\mathcal S}\gate_{Y}(H')$.

Finally, $A$ is unique in its parallelism class if and only if $\orth A=\{x\}$, by definition of $\orth A$.
\end{proof}

We can now prove the proposition:

\begin{proof}[Proof of Proposition~\ref{prop:closed_uncer_complementation}]
The proof has several stages.

\textbf{Setup using Lemma~\ref{lem:proj_to_y}:} If $A$ is a single point, then 
$\orth A=\cuco X$, which is in $\mathfrak F$ by definition.  Hence suppose 
$\diam(A)>0$, and let $H_1,\ldots,H_k$, $x\in A$, $Y\subset\cuco X$, and 
$\mathcal S$ be as in Lemma~\ref{lem:proj_to_y}, so $$\orth 
A=\bigcap_{H'\in\mathcal S}\gate_{Y}(H').$$

Thus, to prove the proposition, it is sufficient to produce a finite 
collection $\mathfrak H$ of hyperplanes $H'$ crossing $A$ so that 
$$\bigcap_{H'\in\mathcal S}\gate_{Y}(H')=\bigcap_{H'\in\mathfrak 
H}\gate_{Y}(H').$$  Indeed, if there is such a collection, then we have shown 
$\orth A$ to be the intersection of finitely many elements of $\mathfrak F_k$, 
whence $\orth A\in\mathfrak F_{k+|\mathfrak H|}$, as required.  Hence suppose 
for a contradiction that for any finite collection $\mathfrak H\subset\mathcal 
S$, we have $$\bigcap_{H'\in\mathcal 
S}\gate_{Y}(H')\subsetneq\bigcap_{H'\in\mathfrak H}\gate_{Y}(H').$$

\textbf{Bad hyperplanes crossing $Y$:}  For each $m$, let $\mathcal H_m$ be the 
(finite) set of hyperplanes $H'$ 
intersecting $\neb_m^A(x)=A\cap\neb_m(x)$ (and hence satisfying $x\in\gate_Y(H')$).

Consider the collection $\mathcal B_m$ of all hyperplanes $W$ such that $W$ 
crosses each element of $\mathcal H_m$ and $W$ crosses $Y$, but $W$ fails to 
cross $\orth A$.  (This means that there exists $j>m$ and some $U\in\mathcal 
H_j$ so that $W\cap U=\emptyset$.)

Suppose that there exists $m$ so that $\mathcal B_n=\emptyset$ for $n>m$.  
Then we can take $\mathcal H_m$ to be our desired set $\mathfrak H$, and we 
are done.    Hence suppose that $\mathcal B_m$ is nonempty for arbitrarily large $m$.  Note 
that if $U\in\mathcal B_m$, then $\neb_m^A(x)$ is parallel into $U$, but there 
exists $j>m$ so that $\neb_j^A(x)$ is not parallel into $U$.  (Here 
$\neb_j^A(x)$ denotes the $j$--ball in $A$ about $x$.)

\textbf{Elements of $\mathcal B_m$ osculating $\orth A$:}  
Suppose that $U\in\mathcal B_m$, so that $\neb_j^A(x)$ is not parallel into $U$ 
for some $j>m$.  Suppose  that $U'$ is a hyperplane separating $U$ from $\orth 
A$.  

Then $U'$ separates $U$ from $x$ so, since $U$ intersects $Y$ and $x\in Y$, 
we have that $U'$ intersects $Y$.  Since $\gate_Y(A)=\{x\}$ is not crossed by 
any hyperplanes, $U'$ cannot cross $A$.  Hence $U'$ separates $U$ from 
$\neb^A_m(x)$, so $\neb^A_m(x)$ is parallel into $U'$ (since it is parallel 
into $U$).  On the other hand, since $U'$ separates $U$ from $\orth A$, $U'$ 
cannot cross $\orth A$, and thus fails to cross some hyperplane crossing $A$.  
Hence $U'\in\mathcal B_m$.  Thus, for each $m$, there exists $U_m\in\mathcal 
B_m$ whose carrier intersects $\orth A$.  Indeed, we have shown that any 
element of $\mathcal B_m$ as close as possible to $\orth A$ has this property.

Hence we have a sequence of radii $r_n$ and hyperplanes 
$U_n$ so that:
\begin{itemize}
     \item $U_n$ crosses $Y$;
     \item $\neb(U_n)\cap\orth A\neq\emptyset$;
     \item $\neb_{r_n}^A(x)$ is parallel into $U_n$ for all $n$;
     \item $\neb_{r_{n+1}}^A(x)$ is not parallel into $U_n$, for all $n$.
\end{itemize}

The above provides a sequence $\{V_n\}$ of hyperplanes so that for each $n$:
\begin{itemize}
 \item $V_n$ crosses $A$;
 \item $V_n$ crosses $U_m$ for $m\geq n$;
 \item $V_n$ does not cross $U_m$ for $m<n$.
\end{itemize}

\noindent Indeed, for each $n$, choose $V_n$ to be a 
hyperplane crossing $\neb_{r_{n+1}}^A(x)$ but not crossing 
$U_n$.  For each $n$, let $\bar V_n$ be one of the two combinatorial 
hyperplanes (parallel to $V_n$) bounding $\neb(V_n)$.

\begin{claim}\label{claim:parallel_into}
For each $n$, the subcomplex $\orth A$ is parallel into $\bar V_n$.
\end{claim}

\begin{proof}[Proof of Claim~\ref{claim:parallel_into}]
Let $H$ be a hyperplane crossing $\orth A$.  Then, by definition of $\orth A$, 
$H$ crosses each hyperplane crossing $A$.  But $V_n$ crosses $A$, so $H$ must 
also cross $V_n$.  Thus $\orth A$ is parallel into $V_n$.
\end{proof}

Next, since $G$ acts on $\cuco X$ cocompactly, it acts with finitely many 
orbits of hyperplanes, so, by passing to a subsequence (but keeping our 
notation), we can assume that there exists a hyperplane $V$ crossing $A$ and 
elements $g_n\in G,n\ge1$ so that $V_n=g_nV$ for $n\ge1$.  For simplicity, we 
can assume $g_1=1$.

Next, we can assume, after moving the basepoint $x\in A$ a single time, that 
$x\in\bar V$, i.e. $\bar V$ is among the $k$ combinatorial hyperplanes whose 
intersection is $Y$.  This assumption is justified by the fact that $\mathfrak 
F$ is closed under parallelism, so it suffices to prove that any given parallel 
copy of $\orth A$ lies in $\mathfrak F$.  Thus we can and shall assume 
$Y\subset\bar V$.

For each $m$, consider the inductively defined subcomplexes $Z_1=\bar V$ and for each $m\ge 2$, $Z_m=\gate_{\bar V}(\gate_{g_1\bar V}(\cdots (\gate_{g_{m-1}\bar V}(g_m\bar 
V))\cdots))$, which is an element of $\mathfrak F$.

\begin{claim}\label{claim:orth_contained}
For all $m\ge1$, we have $\orth A\subseteq Z_m$.
\end{claim}

\begin{proof}[Proof of Claim~\ref{claim:orth_contained}]
Indeed, since $\orth A\subset Y$ by definition, and $Y\subset\bar V$, we have 
$\orth A\subset\bar V$.  On the other hand, for each $n$, 
Claim~\ref{claim:parallel_into} implies that $\orth A$ is parallel into 
$g_n\bar V$ for all $n$, so by induction, $\orth A\subset Z_m$ for all $m\ge 1$, as required.
\end{proof}

\begin{claim}\label{claim:strictly descend}
For all $m\ge1$, we have $Z_m\supsetneq Z_{m+1}$.  
\end{claim}

\begin{proof}[Proof of Claim~\ref{claim:strictly descend}]
For each $m$, the hyperplane $U_m$ crosses $Y$, by construction.  Since 
$Y\subseteq\bar V$, this implies that $U_m$ crosses $\bar V$.  On the other 
hand, $U_m$ does not cross $\bar V_{m+1}$.  This implies that $Z_m\neq 
Z_{m+1}$.  On the other hand, $Z_{m+1}\subset Z_m$ just by definition.
\end{proof}

Let $K_m=\stabilizer_G(V)\cap\bigcap_{n=1}^m\stabilizer_G(V)^{g_n}$; by Lemma~\ref{lem:stab_proj}  $K_m$ has finite index in $\stabilizer_G(Z_m)$.  Thus, since $Z_m\in\mathfrak F$, we 
see that $K_m$ acts on $Z_m$ cocompactly.  Claim~\ref{claim:strictly descend} implies that no $Z_m$ is compact, for 
otherwise we would be forced to have $Z_m=Z_{m+1}$ for some $m$. Since $K_m$ 
acts on $Z_m$ cocompactly, it follows that $K_m$ is infinite for all $m$.

Thus far, we have not used any of the auxiliary hypotheses.  We now explain how 
to derive a contradiction under the weak finite height hypothesis.

\begin{claim}\label{claim:WFH_1}
Suppose that the $G$--action on $\cuco X$ satisfies weak finite height for 
hyperplanes.  Then, after passing to a subsequence, we have $K_m=K_2$ for 
all $m\ge2$. 
\end{claim}

\begin{proof}[Proof of Claim~\ref{claim:WFH_1}]
Let $I\subset\naturals$ be a finite set and let $m=\max I$.  Then 
$\bigcap_{n\in I}\stabilizer_G(V)^{g_i}$ contains $K_m$, and is thus infinite, 
since $K_m$ was shown above to be infinite.  Hence, since $\stabilizer_G(V)$ 
satisfies the weak finite height condition, there exist distinct $m,m'$ so that 
$\stabilizer_G(V)\cap\stabilizer_G(V)^{g_m}
=\stabilizer_G(V)\cap\stabilizer_G(V)^{g_m'}$.  

Declare $m\sim m'$ if $\stabilizer_G(V)\cap\stabilizer_G(V)^{g_m}
=\stabilizer_G(V)\cap\stabilizer_G(V)^{g_m'}$, so that $\sim$ is an equivalence 
relation on $\naturals$.  If any $\sim$--class $[m]$ is infinite, then we can 
pass to the subsequence $[m]$ and assume that 
$\stabilizer_G(V)\cap\stabilizer_G(V)^{g_n}
=\stabilizer_G(V)\cap\stabilizer_G(V)^{g_n'}$ for all $n$.  Otherwise, if every 
$\sim$--class is finite, then there are infinitely many $\sim$--classes, and we 
can pass to a subsequence containing one element from each $\sim$--class.  This 
amounts to assuming that $\stabilizer_G(V)\cap\stabilizer_G(V)^{g_m}
\neq\stabilizer_G(V)\cap\stabilizer_G(V)^{g_m'}$ for all distinct $m,m'$, but 
this contradicts weak finite height, as shown above.

Hence, passing to a subsequence, we can assume that 
$\stabilizer_G(V)\cap\stabilizer_G(V)^{g_m}
=\stabilizer_G(V)\cap\stabilizer_G(V)^{g_m'}$ for all $m,m'\ge2$, and hence 
$K_m=K_2$ for all $m$.
\end{proof}

From Claim~\ref{claim:WFH_1} and the fact that $K_m$ stabilises $Z_m$ for each 
$m$, we have that $K_2$ stabilises each $Z_m$.  Moreover, $K_2$ acts on $Z_m$ 
cocompactly.

The inclusion $Z_{m+1}\hookrightarrow Z_m$ descends to an inclusion $Z_{m+1}/K_2\hookrightarrow Z_m/ K_2$, and 
since the latter spaces are compact, we must have some $M$ such that $Z_m/K_2=Z_M/K_2$ for $M\geq m$.  Hence $Z_m=Z_M$ for 
all $m\geq M$.  This contradicts Claim~\ref{claim:strictly descend}.  Hence the the claimed sequence of $U_m$ cannot exist, 
whence $\orth  A\in\mathfrak F$.

Having proved the proposition under the weak finite height assumption, we now 
turn to the other hypotheses.  Let $\{Z_m\}$ and $\{K_m\}$ be as above.

\begin{claim}\label{claim:NICC}
Suppose that the $G$--action on $\cuco X$ satisfies the NICC and the essential 
index condition.  Then there exists $\ell$ so that 
$\stabilizer_G(Z_m)=\stabilizer_G(Z_\ell)$ for 
all $m\ge\ell$. 
\end{claim}

\begin{proof}[Proof of Claim~\ref{claim:NICC}]
Let $I\subset\naturals$ be a finite set and let $m=\max I$.  Then 
$\bigcap_{n\in I}\stabilizer_G(V)^{g_i}$ contains $K_m$, and is thus infinite, 
since $K_m$ was shown above to be infinite.  Hence, since $\stabilizer_G(V)$ 
satisfies the NICC, there exists $\ell$ so that $K_m$ is commensurable with $K_\ell$ for all $m\ge \ell$.

Hence, by Lemma~\ref{lem:comm_stab}, we have that $\widehat Z_m$ and $\widehat 
Z_\ell$ are parallel for all $m\ge \ell$.  Moreover, since $K_m\leq K_\ell$, Proposition~\ref{prop:CScore} 
implies that we can choose essential cores within their parallelism classes so 
that $\widehat Z_m=\widehat Z_\ell$ for all $m\ge\ell$.

Let $L=\stabilizer_G(\widehat Z_\ell)=\stabilizer_G(\widehat Z_m)$ for all 
$m\ge\ell$.  The essential index condition implies that $\stabilizer_G(Z_m)$ has 
uniformly bounded index in $L$ as $m\to\infty$, so by passing to a further 
infinite subsequence, we can assume that 
$\stabilizer_G(Z_m)=\stabilizer_G(Z_\ell)$ for all $m\ge\ell$ (since $L$ has finitely 
many subgroups of each finite index).
\end{proof}

Claim~\ref{claim:NICC} implies that (up to passing to a 
subsequence), $\stabilizer_G(Z_\ell)=\stabilizer_G(Z_m)$ preserves $Z_m$ 
(and acts cocompactly on $Z_m$) for all $m\ge\ell$.  

Recall that $Z_m\subset Z_\ell$ for $\ell\leq m$.  The inclusion $Z_m\hookrightarrow Z_\ell$ descends to an inclusion 
$Z_m/\stabilizer_G(Z_\ell)\hookrightarrow Z_\ell/\stabilizer_G(Z_\ell)$, and since the latter spaces are compact, there 
exists $M$ such that $Z_m=Z_M$ for all $m\geq M$.  This again contradicts Claim~\ref{claim:strictly descend}.  As before, 
we therefore cannot have the sequences $(U_m),(V_m)$ with the given properties, and hence $\orth A\in\mathfrak F$.

\textbf{Applying the factor system assumption:}  If $\mathfrak F$ is a factor system, then since $Z_m\in\mathfrak F$ for all $m$, and $Z_m\supsetneq Z_{m+1}$ for all $m$, we have an immediate contradiction, so $\orth A\in\mathfrak F$.

\textbf{Conclusion:}  We have shown that under any of the additional hypotheses, 
$\orth A\in\mathfrak F$ when $A\subset\cuco X$ is a convex subcomplex.  This 
holds in particular if $A\in\mathfrak F$.  
\end{proof}

The preceding proposition combines with earlier facts to yield:

\begin{cor}[Ascending and descending chains]\label{cor:asc_chain}
Let $G$ act properly and cocompactly on $\cuco X$, satisfying any of the 
hypotheses of Proposition~\ref{prop:closed_uncer_complementation}.  Suppose 
that for all 
$N\geq\infty$, there exists a $0$--cube $x\in\cuco X$ so that $x$ lies in at 
least $N$ elements of $\mathfrak F$.  Then there exist sequences 
$(F_i)_{i\ge1},(F_i')_{i\ge1}$ of subcomplexes in $\mathfrak F$ so that all of 
the following hold for all $i\ge1$:
\begin{itemize}
     \item $F_i\subsetneq F_{i+1}$;
     \item $F_i'\supsetneq F'_{i+1}$;
     \item $F_i'=\orth F_i$.
\end{itemize}
Moreover, there exists a $0$--cube $x$ that lies in each $F_i$ and each 
$F_i'$.
\end{cor}

\begin{proof}
Lemma~\ref{lem:chain_find}, cocompactness, and $G$--invariance of $\mathfrak F$ provide a sequence $(F_i)$ in $\mathfrak F$ and a 
point $x$ so that $x\in F_i$ for all $i$ and either $F_i\subsetneq F_{i+1}$ for 
all $i$, or $F_i\supsetneq F_{i+1}$ for all $i$.  For each $i$, let 
$F_i'=\phi_{F_i}(\{x\}\times\orth F_i)$.  
Proposition~\ref{prop:closed_uncer_complementation} implies that each 
$F_i'\in\mathfrak F$, and Lemma~\ref{lem:contravariant} implies that $(F_i')$ 
is an ascending or descending chain according to whether $(F_i)$ was descending 
or ascending.  Assume first that $F_i\subsetneq F_{i+1}$ for all $i$.  Now, if 
$F_i'=\orth F_i=\orth F_{i+1}=F_{i+1}'$, then by 
Corollary~\ref{cor:f_characterise}, we have $F_i=F_{i+1}$, a contradiction.  
Hence $(F'_i)$ is properly descending, i.e. $F_i'\supsetneq F_{i+1}'$ for all 
$i$.  The case where $(F_i)$ is descending is identical.  This completes the 
proof.
\end{proof}

\section{Proof of Theorem~\ref{thmi:main}}\label{sec:proof_of_main_theorem}
We first establish the setup.  Recall that $\cuco X$ is a proper CAT(0) cube 
complex with a proper,
cocompact action by a group $G$. We denote the hyperclosure by $\mathfrak F$; 
our goal is to prove that there exists $N<\infty$ so
that each $0$--cube of $\cuco X$ is contained in at most $N$
elements of $\mathfrak F$, under any of the three additional hypotheses of 
Theorem~\ref{thmi:main}.

If there is no such $N$, 
then Corollary~\ref{cor:asc_chain} implies that there 
exists a $0$--cube $x\in\cuco X$
and a sequence $(F_i)_{i\geq1}$ in $\mathfrak F$ so that $x\in F_i\subsetneq F_{i+1}$ for each $i\geq1$.  For the sake of brevity, given any subcomplex $E\ni x$, let $\orth E$ denote the orthogonal
complement of $E$ \emph{based at $x$}.  Corollary~\ref{cor:asc_chain} also says 
that $\orth F_i\in\mathfrak F$ for 
all $i$ and $\orth F_i\supsetneq\orth F_{i+1}$ for all $i$. Proposition~\ref{prop:f_n_properties} shows that $\stabilizer_G(\orth F_i)$ acts on $\orth F_i$ cocompactly for all $i$.

Let $U=\bigcup_iF_i$ and let $I=\bigcap_i\orth F_i$, and note that $\orth U=I$ 
and $\orth I=U$.  In particular, $\orth{\orth U}=\orth I=U$ and $\orth{\orth 
I}=\orth U=I$.  From here, we can now prove our main theorem:

\begin{proof}[Proof of Theorem~\ref{thmi:main}]
We have already proved the theorem under the rotation hypothesis, in 
Corollary~\ref{cor:rotation_FS}.  Hence suppose that either weak 
finite height holds or the 
NICC and essential index 
conditions both hold, so that 
Proposition~\ref{prop:closed_uncer_complementation} implies that 
$U,I\in\mathfrak F$.

By Corollary~\ref{cor:f_characterise}, we have 
compact convex subcomplexes $D,E$ with $U=\orth D$ and $I=\orth E$.  Moreover, 
we can take $D\subset I$ and $E\subset U$.  Now, 
Corollary~\ref{cor:f_characterise}, 
Theorem~\ref{thm:factor_system_compact_orthocomp}, and 
Proposition~\ref{prop:closed_uncer_complementation} provide, for each 
$i\geq1$, a compact, convex subcomplex $C_i$, containing $x$ and contained in
$F_i$, so that $\orth C_i=\orth F_i$.  

Let $C'_1=C_1$.  For $i\ge2$, let $C_i'=Hull(\cup_{j<i}C_j)$.  Note that $C_i'$ is a compact, convex subcomplex contained in 
$F_i$, so $C_i\subseteq C_i'\subseteq F_i$ for $i\geq1$.
%

By Lemma~\ref{lem:contravariant}, for each $i$, 
$\orth F_i\subseteq\orth{(C'_i)}\subseteq\orth{C}_i=\orth F_i$ since
$C_i\subseteq C_i'$.  Hence $(C'_i)_{i\geq1}$ is an ascending
sequence of convex, compact subcomplexes, containing $x$, with $\orth 
{(C_i')}=\orth F_i$ for all $i$. 

Note that $\bigcap_i\orth{(C_i')}=\bigcap_i\orth F_i=I$.  However, 
$I=\orth{\left(\bigcup C_i'\right)}$. But, $\bigcup C_i'$ cannot be compact 
since $C_i'\subsetneq C_{i+1}'$, and by Corollary~\ref{cor:f_characterise}, we 
can choose $E\subseteq \bigcup C_i'$. Thus $E\subseteq C_R'$ for some $R$. But 
by Lemma~\ref{lem:contravariant} this means that $I=\orth{E}\supseteq 
\orth{\left(C_R'\right)}=\orth{F_R}$, a contradiction.  Thus, $\mathfrak F$ 
must have finite multiplicity, as desired.
\end{proof}

We show now that for cube complexes that admit geometric actions, 
having a factor system implies the NICC for hyperplanes and essential index conditions for hyperplanes. Thus, any proof that $\mathfrak F$ forms a factor system for all cocompact cubical groups must necessarily show that any group acting geometrically on a CAT(0) cube complex satisfies these conditions.

\begin{thm}\label{thm:converse} Let $\cuco X$ be a CAT(0) cube complex admitting 
a geometric action by a group $G$. If $\mathfrak F$ is a factor 
system, then $G$ satisfies NICC for hyperplanes 
and the essential index condition.\end{thm}

\begin{proof}
Suppose that $\mathfrak F$ is a factor system and at most $N$ elements of 
$\mathfrak F$ can contain any given $x\in\cuco X^{(0)}$. Then for any $A, B\in\mathfrak F_1$, there are at most $N$ elements of $\mathfrak F$ which can 
contain $\widehat{F}$, the $\stabilizer_G(F)$--essential core of $F$. In particular, there are at most $N$ distinct 
$\stabilizer_G(\widehat{F})$--translates of $F$. Thus, 
$[\stabilizer_G(\widehat{F}),\stabilizer_G(F)]\le N$, verifying the essential index 
condition.

To verify NICC for hyperplanes, let $H$ be a hyperplane and let $K=\stabilizer_G(H)$. Let $\{g_i\}_{i=1}^{\infty}$ be sequence of distinct elements of $G$ so that for $n\ge1$, the subgroup $K\cap\bigcap_{i=1}^nK^{g_i}$ is infinite.

Consider the hyperplane $H$, notice that $K^{g_i}$ is the stabilizer of $g_i H$. 
Now, consider $F_1=\gate_H(g_1H)$ and inductively define 
$F_k=\gate_{F_{k-1}}(\gate_H(g_kH))$. Since $\mathfrak F$ is a factor system, 
the set of $G$--translates of $F_{k-1}$ and $\gate_H(g_kH)$ have finite 
multiplicity for all $k\ge 2$, and so we can apply the argument of 
Lemma~\ref{lem:stab_proj} and induction to conclude that $\stabilizer_G(F_k)$ is 
commensurable with $G_k=K\cap\bigcap_{i=1}^k K^{g_k}$, which is infinite by 
assumption.

Since $\mathfrak F$ is a factor system, there must be some $\ell$ so that for all $k\ge \ell$, $F_k=F_\ell$. In this case, $G_k$ and $G_\ell$ are commensurable for all $k\ge \ell$, and in particular $G_k$ and $G_{k'}$ are commensurable for all $k, k'\ge \ell$, and thus $G$ satisfies NICC for hyperplanes.
\end{proof}

\section{Factor systems and the simplicial boundary}\label{sec:fullvis}
Corollary~\ref{cori:visible} follows from Theorem~\ref{thmi:main}, Proposition~\ref{prop:visibility} and~\cite[Lemma~3.32]{Hagen:boundary}.  Specifically, the first two statements provide a combinatorial geodesic ray representing each boundary simplex $v$, and when $v$ is a $0$--simplex,~\cite[Lemma~3.32]{Hagen:boundary} allows one to convert the combinatorial geodesic ray into a CAT(0) ray.  Proposition~\ref{prop:visibility} is implicit in the proof of~\cite[Theorem 10.1]{DHS:HHS_IV}; we give a streamlined proof here.  

\begin{prop}\label{prop:visibility}
Let $\cuco X$ be a CAT(0) cube complex with a factor system $\mathfrak
F$.  Then each simplex $\sigma$ of $\simp\cuco X$ is \emph{visible}, i.e. there exists a combinatorial
geodesic ray $\alpha$ such that the set of hyperplanes intersecting
$\alpha$ is a boundary set representing the simplex $\sigma$.
\end{prop}

\begin{rem}\label{rem:visible}
Proposition~\ref{prop:visibility} does not assume anything about group actions on $\cuco
X$, but instead shows that the existence of an invisible boundary
simplex is an obstruction to the existence of a factor system.  The
converse does not hold: counterexamples can be constructed by beginning with a single combinatorial
ray, and gluing to the $n^{th}$ vertex a finite staircase $S_n$, along a single 
vertex.  The staircase $S_n$ is obtained from $[0,n]^2$ by deleting all squares 
that are strictly above the diagonal joining $(0,0)$ to $(n,n)$.  In this case, 
$\mathfrak F$ has unbounded multiplicity, and any factor system must contain all 
elements of $\mathfrak F$ exceeding some fixed threshold diameter, so the 
complex cannot have a factor system.
\end{rem}

\begin{proof}[Proof of Proposition~\ref{prop:visibility}]
Let $\sigma$ be a simplex of $\simp\cuco X$.  Let $\sigma'$ be a
maximal simplex containing $\sigma$, spanned by $v_0,\ldots,v_d$.
The existence of $\sigma'$ follows from~\cite[Theorem
3.14]{Hagen:boundary}, which says that maximal simplices exist since $\cuco X$ is finite-dimensional (otherwise, it could not have a factor system).  By Theorem~3.19 of~\cite{Hagen:boundary}, which says that maximal simplices are visible, $\sigma'$ is visible, i.e. there exists a combinatorial geodesic ray
$\gamma$ such that the set $\mathcal H(\gamma)$ of hyperplanes
crossing $\gamma$ is a boundary set representing $\sigma'$.  We will
prove that each $0$--simplex $v_i$ is visible.  It then follows from~\cite[Theorem~3.23]{Hagen:boundary} that any
face of $\sigma'$ (hence $\sigma$) is visible.

Let $\cuco Y$ be the convex hull of $\gamma$.  The set of
hyperplanes crossing $\cuco Y$ is exactly $\mathcal H(\gamma)$.  Since
$\cuco Y$ is convex in $\cuco X$, Lemma 8.4 of~\cite{BHS:HHS_I}, which provides an induced factor system on convex subcomplexes of cube complexes with factor systems, implies
that $\cuco Y$ contains a factor system.

By Theorem~3.10 of~\cite{Hagen:boundary}, we can write $\mathcal
H(\gamma)=\bigsqcup_{i=1}^d\mathcal V_i,$ where each $\mathcal V_i$
is a minimal boundary set representing the $0$--simplex $v_i$.
Moreover, up to reordering and discarding finitely many
hyperplanes (i.e. moving the basepoint of $\gamma$) if necessary,
whenever $i<j$, each hyperplane $H\in\mathcal V_j$ crosses all but
finitely many of the hyperplanes in $\mathcal V_i$.  

For each $1\leq i\leq d$, minimality of $\mathcal V_i$ provides a
sequence of hyperplanes $(V_n^i)_{n\geq0}$ in $\mathcal V_i$ so that
$V^i_n$ separates $V_{n\pm1}^i$ for $n\geq1$ and so that any other
$U\in\mathcal V_i$ separates $V^i_m,V^i_n$ for some $m,n$, by the proof 
of~\cite[Lemma 3.7]{Hagen:boundary} or~\cite[Lemma 
B.6]{ChatterjiFernosIozzi} (one may
have to discard finitely many hyperplanes from $\mathcal V_i$ for
this to hold; this replaces $\gamma$
with a sub-ray and shrinks $\cuco Y$).

We will show that, after discarding finitely many hyperplanes from
$\mathcal H(\gamma)$ if necessary, every element of $\mathcal V_i$
crosses every element of $\mathcal V_j$, whenever $i\neq j$. Since
every element of $\mathcal V_i$ either lies in $(V^i_n)_n$ or
separates two elements of that sequence, it follows that $U$ and $V$
cross whenever $U\in\mathcal V_i,V\in\mathcal V_j$ and $i\neq j$.
Then, for any $i$, choose $n\geq0$ and let $H=\bigcap_{j\neq
i}V_n^j$. Projecting $\gamma$ to $H$ yields a geodesic ray in $\cuco
Y$, all but finitely many of whose dual hyperplanes belong to
$\mathcal V_i$, as required.  Hence it suffices to show that $V_n^i$
and $V_m^j$ cross for all $m,n$ whenever $i\neq j$.

Fix $j\leq d$ and $i<j$.  For each $n\geq0$, let $m(n)\geq0$ be
minimal so that $V_{m(n)}^j$ fails to cross $V_n^i$.  Note that we
may assume that this is defined: if $V_n^i$ crosses all $V^j_m$,
then, since $V^j_m$ crosses all but finitely many of the hyperplanes
from $\mathcal V_i$, it crosses $V_k^i$ for $k>>n$. Since it also
crosses $V_n^i$, it must also cross $V_r^i$ for all $n\le r\le k$.
By discarding $V_k^i$ for $k\leq n$ we complete the proof.  Now
suppose that $m(n)$ is bounded as $n\to\infty$. Then there exists
$N$ so that $V^i_n,V^j_m$ cross whenever $m,n\geq N$, and we are
done, as before.

Hence suppose that $m(n)\to\infty$ as $n\to\infty$.  In other words, for all $m\geq0$, there exists $n\geq0$ so that $V_m^j$ crosses $V_k^i$ if and only if $k\geq n$.  Choose $M\gg0$ and choose $n$ maximal with $m(n)<M$.  Then all of the hyperplanes $V^j_{m(k)}$ with $k\leq n$ cross $V^i_k,\ldots,V^i_n$ but do not cross $V^i_t$ for $t<k$.  Hence the subcomplexes $\gate_{V^i_n}(V^i_k),k\leq n$ are all different: $\gate_{V^i_n}(V^i_k)$ intersects $V^j_{m(k)}$ but $\gate_{V^i_n}(V^i_{k-1})$ does not.  On the other hand, since $V^i_k$ separates $V^i_\ell$ from $V^i_n$ when $\ell<k<n$, every hyperplane crossing $V^i_n$ and $V^i_\ell$ crosses $V^i_k$, so $\gate_{V_n^i}(V^i_k)\cap\gate_{V_n^i}(V^i_\ell)\neq\emptyset$.  Thus the factor system on $\cuco Y$ has multiplicity at least $n$.  But since $m(n)\to\infty$, we could choose $n$ arbitrarily large in the preceding argument, violating the definition of a factor system.
\end{proof}

\begin{proof}[Proof of Corollary~\ref{cori:orthant}]
If $\gamma$ is a CAT(0) geodesic, then it can be approximated, up to
Hausdorff distance depending on $\dimension\cuco X$, by a
combinatorial geodesic, so assume that $\gamma$ is a combinatorial geodesic ray.
By Corollary~\ref{cori:visible}, the simplex of $\simp\cuco X$
represented by $\gamma$ is spanned by $0$--simplices
$v_0,\ldots,v_d$ with each $v_i$ represented by a combinatorial
geodesic ray $\gamma_i$.  Theorem~3.23 of~\cite{Hagen:boundary} says that $\cuco 
X$ contains a cubical orthant $\prod_i\gamma_i'$, where each $\gamma_i'$ 
represents $v_i$.  Hence $Hull(\cup_i\gamma'_i)=\prod_i Hull(\gamma_i)$.  Up to 
truncating an initial subpath of $\gamma$, we have that $\gamma$ is 
parallel into $Hull(\cup_i\gamma_i')$ (and thus lies in a finite neighbourhood 
of it).  The projection of the original CAT(0) geodesic approximated by $\gamma$ 
to each $Hull(\gamma_i')$ is a CAT(0) geodesic representing $v_i$.  The product 
of these geodesics is a combinatorially isometrically embedded $(d+1)$-dimensional orthant subcomplex of $\cuco Y$ containing (the 
truncated) CAT(0) geodesic in a regular neighbourhood.
\end{proof}

In the presence of a proper, cocompact group action, we can achieve full 
visibility under slightly weaker conditions than those that we have shown 
suffice to obtain a factor system:

\begin{prop}\label{prop:visibility_group_action}
Let $\cuco X$ be a proper CAT(0) cube complex on which the group $G$ acts 
properly and cocompactly.  Suppose that the action of $G$ on $\cuco X$ 
satisfies NICC for hyperplanes. Then each simplex $\sigma$ of $\simp\cuco X$ is \emph{visible}, i.e. there exists a combinatorial
geodesic ray $\alpha$ such that the set of hyperplanes intersecting
$\alpha$ is a boundary set representing the simplex $\sigma$.
\end{prop}

\begin{proof}
We adopt the same notation as in the proof of 
Proposition~\ref{prop:visibility}.  As in that proof, if $\simp\cuco X$ 
contains an invisible simplex, then we have two infinite sets 
$\{V_i\}_{i\geq0},\{H_j\}_{j\geq0}$ of hyperplanes with the following 
properties:
\begin{itemize}
     \item for each $i\geq1$, the hyperplane $H_i$ separates  $H_{i-1}$ from 
$H_{i+1}$;
    \item for each $j\geq 1$, the hyperplane $V_j$ separates  $V_{j-1}$ from 
$V_{j+1}$;
    \item there is an increasing sequence $(i_j)$ so that for all $j$, $V_j$ crosses $H_i$ if and only if $i\leq i_j$.
\end{itemize}
This implies that for all $i\ge1$, the subcomplex $F_i=\gate_{H_0}(\gate_{H_1}(\cdots(\gate_{H_{i-1}}(H_i))\cdots))$ is 
unbounded.  Since $\stabilizer_G(F_i)$ acts cocompactly, by 
Proposition~\ref{prop:f_n_properties}, $\stabilizer_G(F_i)$ is 
infinite. By Lemma~\ref{lem:stab_proj}, $\stabilizer_G(F_i)$ is commensurable with $K_i=\bigcap_{j=1}^i \stabilizer_G(H_j)$, and so by NICC, there exists $N$ so that $K_i$ is commensurable with $K_N$ for all $i\ge N$. Thus, after passing to a subsequence, we 
see that for all $i$, the $K_i$--essential core of $F_i$ is a fixed 
nonempty (indeed, unbounded) convex subcomplex $\widehat F$ of $H_0$.

Now, for each $j$, the hyperplane $V_j$ cannot cross $\widehat F$, because 
$\widehat F$ lies in $F_i$ for all $i$, and $V_j$ fails to cross 
$H_i$ for all sufficiently large $i$.  Moreover, this shows that $\widehat F$ 
must lie in the halfspace associated to $V_j$ that contains $V_{j+1}$.  But 
since this holds for all $j$, we have that $\widehat F$ is contained in an 
infinite descending chain of halfspaces, contradicting that $\widehat 
F\neq\emptyset$.
\end{proof}

\bibliography{Cube_Projections}
\bibliographystyle{amsalpha}
\end{document}